\documentclass[a4paper, 12pt, twoside]{article}

\usepackage[top=3cm, bottom=3cm, left=3cm, right=3cm]{geometry}

\usepackage{array}

\usepackage{amsmath}
\usepackage{amsthm}
\usepackage{amssymb}

\usepackage{mathtools}
\usepackage{stmaryrd}

\usepackage[mathcal]{euscript}

\usepackage[utf8]{inputenc}
\usepackage[T1]{fontenc}

\usepackage[english]{babel}

\usepackage{cite}

\usepackage{mathptmx}

\usepackage[inline]{enumitem}

\usepackage[normalem]{ulem}
\usepackage{xcolor}

\usepackage[hidelinks, bookmarks, bookmarksnumbered, pdfstartview={XYZ null null 1.00}]{hyperref}

\theoremstyle{plain}
\newtheorem{theo}{Theorem}[section]
\newtheorem{lemm}[theo]{Lemma}
\newtheorem{coro}[theo]{Corollary}
\newtheorem{prop}[theo]{Proposition}
\theoremstyle{definition}
\newtheorem{defi}[theo]{Definition}
\newtheorem{remk}[theo]{Remark}
\newtheorem{expl}[theo]{Example}

\DeclareMathOperator{\id}{Id}

\newcommand{\rhod}{\rho_{\mathrm d}}
\newcommand{\rhop}{\rho_{\mathrm p}}

\DeclarePairedDelimiter\ceil{\lceil}{\rceil}

\DeclarePairedDelimiter\abs{\lvert}{\rvert}
\DeclarePairedDelimiter\norm{\lVert}{\rVert}

\numberwithin{equation}{section}

\let\originalleft\left
\let\originalright\right
\renewcommand{\left}{\mathopen{}\mathclose\bgroup\originalleft}
\renewcommand{\right}{\aftergroup\egroup\originalright}

\begin{document}

\setlist[enumerate]{label={\textnormal{(\alph*)}}, ref={(\alph*)}, leftmargin=0pt, itemindent=*}

\title{On the gap between deterministic and probabilistic joint spectral radii for discrete-time linear systems\thanks{This research was partially supported by the iCODE Institute, research project of the IDEX Paris-Saclay. The second author was partially supported by the public grant number ANR-10-CAMP-0151-02-FMJH as part of the ``Programme des Investissements d'Avenir''.}}
\author{Yacine Chitour\thanks{Universit\'e Paris-Saclay, CNRS, CentraleSup\'elec, Laboratoire des signaux et syst\`emes, 91190, Gif-sur-Yvette, France.}, Guilherme Mazanti\footnotemark[\value{footnote}]~\thanks{Inria,  France.}, Mario Sigalotti\footnotemark[\value{footnote}]~\thanks{
Sorbonne Universit\'e, Universit\'e de Paris, CNRS, Inria,  Laboratoire Jacques-Louis Lions, 75005 Paris, France.}}

\maketitle

\hypersetup{pdftitle={On the gap between deterministic and probabilistic joint spectral radii for discrete-time linear systems}, pdfauthor={Yacine Chitour, Guilherme Mazanti, Mario Sigalotti}}

\begin{abstract}
Given a discrete-time linear switched system $\Sigma(\mathcal A)$ associated with
a finite set $\mathcal A$ of matrices, we consider 
the measures of its asymptotic behavior given by, on the one hand, its deterministic joint spectral radius $\rhod(\mathcal A)$ and, on the other hand, 
its probabilistic joint spectral radii $\rhop(\nu,P,\mathcal A)$ for Markov random switching signals with transition matrix $P$ and a corresponding invariant probability $\nu$. Note that $\rhod(\mathcal A)$ is larger than or equal to $\rhop(\nu,P,\mathcal A)$ for every pair $(\nu, P)$. In this paper, we investigate the cases of equality of $\rhod(\mathcal A)$ with either a single $\rhop(\nu,P,\mathcal A)$ or with the supremum of $\rhop(\nu,P,\mathcal A)$ over $(\nu,P)$ and we aim at characterizing the sets $\mathcal A$ for which such equalities may occur.
\end{abstract}



\medskip

{\small \noindent \textcolor{blue}{\emph{This paper was first published in \emph{Linear Algebra and its Applications}, 613:24–45, 2021. With respect to the published version, this version provides an additional remark (Remark~\ref{remk:Lyap-exp}) and a more precise proof of Lemma~\ref{LemmMorris}. All modifications with respect to the published version are given in blue. The authors are very grateful to Matteo Della Rossa for pointing out the imprecisions in the previous version of the proof of Lemma~\ref{LemmMorris}.}}}

\section{Introduction}

In this paper, we consider discrete-time switched linear systems of the form
\begin{equation}
\label{SwitchedSystem}
\Sigma(\mathcal A): \qquad x_{k+1} = A_{\sigma(k)} x_k, \qquad \sigma \in \mathfrak S,\quad k \in \mathbb N,
\end{equation}
where $d$ and $N$ are positive integers, $x_k \in \mathbb R^d$, $\mathfrak S$ is the set of the set of all maps $\sigma: \mathbb N \to \{1, \dotsc, N\}$, and $\mathcal A = (A_1, \dotsc, A_N)$ is an $N$-tuple of $d \times d$ matrices with real coefficients.

Switched systems model the behavior of a continuous variable $x$ whose dynamics may change over time according to the value of a discrete variable $\sigma$. These models are useful for several applications, ranging from air traffic control, electronic circuits, and automotive engines to chemical processes and population models in biology. This wide field of applications, together with the interesting mathematical questions arising from their analysis, justify the extensive literature on switched systems, which have been studied from the point of view of both deterministic and random switching \cite{Shorten2007Stability, Liberzon2003Switching, Sun2005Switched, Lin2009Stability, Costa2013Continuous, Costa2005Discrete}. A commonly used point of view on the switching signal $\sigma$, which we adopt in this paper, is to consider it as an uncertainty or perturbation acting on the system, the goal being thus to provide properties of the system independent of a particular choice of $\sigma$.

We are interested in describing the asymptotic behavior of $\Sigma(\mathcal A)$. For a given $\sigma \in \mathfrak S$, the asymptotic behavior of the corresponding non-autonomous linear system is measured by the quantity $\rho(\sigma)$ defined by
\begin{equation}
\label{eq:intro-rho-sigma}
\rho(\sigma) = \limsup_{n \to \infty} \norm{A_{\sigma(n)} \dotsm A_{\sigma(1)}}^{1/n}.
\end{equation}
Indeed, $\rho(\sigma) < 1$ if and only if all trajectories of the non-autonomous system $x_{k+1} = A_{\sigma(k)} x_k$ converge exponentially to the origin.

In order to capture the asymptotic behavior of $\Sigma(\mathcal A)$, we must formulate some condition which is independent of the choice of $\sigma \in \mathfrak S$. There exist two main approaches to proceed. The first one is deterministic and consists in considering the \emph{joint spectral radius} $\rhod(\mathcal A)$ of $\mathcal A$, defined as the supremum of $\rho(\sigma)$ over all $\sigma \in \mathfrak S$. Since its introduction in \cite{RotaStrang1960} and after the seminal paper \cite{Daubechies1992Sets}, it has been extensively studied in the computer science and control theory communities (see, e.g., the monograph \cite{Jungers2009Joint}).

The other approach to handle the asymptotic behavior of $\Sigma(\mathcal A)$ is probabilistic and amounts to considering a probability measure $\mu$ on $\mathfrak S$ and hence $\sigma \mapsto \rho(\sigma)$ as a random variable. We may then consider as a \emph{probabilistic joint spectral radius} the expected value of $\rho(\sigma)$ with respect to the probability law $\mu$, which we denote by $\rhop(\mu, \mathcal A)$. There exists a vast literature devoted to the properties of products of random matrices, and we refer the reader to \cite{Arnold1998Random, Colonius2014Dynamical, Crisanti1993Products} for more details. A major result in this field has been obtained in \cite{Furstenberg1960Products} and provides general conditions on $\mu$ under which $\rho(\sigma) = \rhop(\mu, \mathcal A)$ on a set of $\mu$ probability~$1$.

The interest in considering $\rhod(\mathcal A)$ and $\rhop(\mu, \mathcal A)$ comes from the stability analysis of \eqref{SwitchedSystem}. Indeed, $\rhod(\mathcal A) < 1$ if and only if \eqref{SwitchedSystem} is uniformly exponentially stable \cite{Jungers2009Joint}, whereas, under the conditions of \cite{Furstenberg1960Products}, $\rhop(\mu, \mathcal A) < 1$ if and only if $\mu$-almost every trajectory of \eqref{SwitchedSystem} converges exponentially to the origin.

In this paper, we aim at understanding the relations between the deterministic and the probabilistic approaches. The deterministic measure of stability $\rhod(\mathcal A)$ characterizes the worst possible behavior over all $\sigma \in \mathfrak S$, while the probabilistic counterpart $\rhop(\mu, \mathcal A)$ provides the average behavior for $\sigma \in \mathfrak S$ corresponding to the probability measure $\mu$. As a consequence, the deterministic approach provides a more conservative estimate of the asymptotic behavior of the system than the probabilistic one, in the sense that
\begin{equation}
\label{IntroInequality-NuP}
\rhop(\mu, \mathcal A) \leq \rhod(\mathcal A).
\end{equation}

A natural question is then to understand under which conditions on $\mathcal A$ and $\mu$ the inequality in \eqref{IntroInequality-NuP} is strict. Furthermore, for practical and modeling purposes, it is important to understand whether, given a family of probability measures $\{\mu_\ell\}_{\ell \in \mathcal I}$, the strict inequality $\sup_{\ell \in \mathcal I} \rhop(\mu_\ell, \mathcal A) < \rhod(\mathcal A)$ holds true. Regarding the first question, it is known that there always exists a measure $\mu$ such that equality holds in \eqref{IntroInequality-NuP} (see, for instance, \cite{Morris2013Mather}, where such measures are referred to as \emph{maximizing measures}). At such a level of generality, a handy characterization of maximizing measures cannot be expected. This is why we restrict our attention to the family $\mathfrak M$ of probability measures on $\mathfrak S$ obtained from discrete-time shift-invariant Markov chains %
% on the state space $\{1, \dotsc, N\}$
%with a transition matrix $P$ and a corresponding invariant probability $\nu$
%
and reformulate the previous two questions as follows: under which conditions on $\mathcal A$ do we have
\begin{enumerate}[label={(Q\arabic*)}]
\item\label{Question1} equality between $\rhop(\mu, \mathcal A)$ and $\rhod(\mathcal A)$ for a given $\mu \in \mathfrak M$?
\item\label{Question2} equality between $\sup_{\mu \in \mathfrak M} \rhop(\mu, \mathcal A)$ and $\rhod(\mathcal A)$?
%\item Under which conditions on $\mu \in \mathfrak M$ and $\mathcal A$ does equality hold in \eqref{IntroInequality-NuP}?
%\item Under which conditions on $\mathcal A$ is $\sup_{\mu \in \mathfrak M} \rhop(\mu, \mathcal A)$ equal to $\rhod(\mathcal A)$?
\end{enumerate}
%Assume that $\sigma$ is generated by a discrete-time Markov chain on the state space $\{1, \dotsc, N\}$ with a transition matrix $P$ and a corresponding invariant probability $\nu$. Under which conditions on $\mathcal A$ does equality hold in \eqref{IntroInequality-NuP} for a given $\mu = (\nu, P)$ or for some $\mu = (\nu, P)$, where $(\nu, P)$ define a Markov chain?

Notice that  the condition $\sup_{\mu \in \mathfrak M} \rhop(\mu, \mathcal A) < 1$ is related to the almost sure stability of the system uniformly with respect to the Markov process, a stability property first considered in \cite{Gurvits1995Stability} in the case of Markov chains with positive transition probabilities. Other stability notions have also been considered for \eqref{SwitchedSystem}, such as periodic stability, meaning stability for all \emph{periodic} signals $\sigma \in \mathfrak S$, or mean square stability. Several works explore relations between these different notions, see, e.g., \cite{Gurvits1995Stability, Fang1994Almost, Feng1992Stochastic, Dai2011Periodically, Costa2005Discrete, Dai2011Weakly}. In particular, \cite{Dai2011Periodically} establishes a probabilistic version of the finiteness conjecture, i.e., if \eqref{SwitchedSystem} is periodically stable, then $\rhop(\mu, \mathcal A) < 1$ for every $\mu \in \mathfrak M$.

Another interesting fact is that the quantities $\rhod(\mathcal A)$ and of $\rhop(\mu, \mathcal A)$ for $\mu \in \mathfrak M$ could be equivalently computed by replacing the norm $\norm{\cdot}$ in \eqref{eq:intro-rho-sigma} by the spectral radius. In the deterministic case, this result is known as the Berger--Wang formula or also as the Joint Spectral Radius Theorem \cite{Jungers2009Joint}, and it has been extended to the Markovian setting in \cite{Kozyakin2014Berger, Dai2014Robust}.

In order to describe the main results of our paper, let us identify a measure $\mu \in \mathfrak M$ with the pair $(\nu, P)$, where $P$ is the transition matrix of the Markov chain corresponding to $\mu$ and $\nu$ is its (invariant) initial probability. In particular, we write $\rhop(\nu, P, \mathcal A)$ for $\rhop(\mu, \mathcal A)$. Our main result concerning \ref{Question1} (see Theorem \ref{MainTheoFixedP}) establishes that a necessary and sufficient condition for equality is that $\rhod(\mathcal A) = \rho(A_{i_k} \dotsm A_{i_1})^{1/k}$ for every $(i_1, \dotsc, i_k)$ that corresponds to a cycle in the directed weighted graph determined by $P$ such that $\nu_{i_1} > 0$. The necessity follows from results provided in \cite{Morris2013Mather}, whereas, for sufficiency, we consider first the particular case where $\mathcal A$ is irreducible and $P$ is strongly connected (see Lemma~\ref{Lemm-PSC-AIrreducible}). Irreducibility implies in particular the existence of a Barabanov norm for $\mathcal A$ (see Definition~\ref{DefiBarabanov}), which is an important tool in our proof. We then generalize the result to the case of reducible $\mathcal A$ (see Lemma~\ref{Lemm-PSC}) by a suitable block decomposition of the matrices in $\mathcal A$ and the fact that $\rhop(\nu, P, \mathcal A)$ and $\rhod(\mathcal A)$ can be read on the diagonal blocks of the decomposed matrix (cf.\ \cite{Jungers2009Joint, Gerencser2008Stability}). Finally, the general case for $P$ can be obtained by using a classical block decomposition of stochastic matrices.

The equivalence established in Theorem~\ref{MainTheoFixedP} can be further characterized in terms of simultaneous similarity of the matrices $\rhod(\mathcal A)^{-1} A_i$, $i \in \{1, \dotsc, N\}$, to orthogonal matrices, under some additional assumptions on $\mathcal A$ and $P$ (Proposition \ref{PropPPositiveDiagonal}). The latter characterization is based on the description of matrix semigroups with constant spectral radius from \cite{Protasov2017Matrix}.

Our next main result, Theorem~\ref{MainTheo}, concerns \ref{Question2} and states that equality is equivalent to the existence of a family of pairwise distinct indices $i_1, \dotsc, i_k \in \{1, \dotsc, N\}$ such that $\rhod(\mathcal A) = \rho(A_{i_1} \dotsm A_{i_k})^{1/k}$. This corresponds to the case where the worst behavior of the system is attained by a periodic $\sigma$ with no repetition of indices on a period. This property is reminiscent of the \emph{finiteness property}, except for the fact that, in the finiteness property, repetition of indices is allowed. We recall that the finiteness property is known to hold only for a proper subclass of  $N$-tuples $\mathcal A$ \cite{Blondel2003Elementary, Bousch2002Asymptotic},
contrarily to what had been earlier conjectured \cite{Lagarias1995Finiteness}.
%It was conjectured in \cite{Lagarias1995Finiteness} that the finiteness property would hold for every $N$-tuple $\mathcal A$, but this has been disproved in \cite{Blondel2003Elementary, Bousch2002Asymptotic}.
By applying a standard lifting argument of Markov chains of higher order to Markov chains of order one, we generalize the equivalence stated in Theorem~\ref{MainTheo} by providing the following characterization of the finiteness property: a $N$-tuple $\mathcal A$ satisfies the finiteness property if and only if there exist $m\ge 1$ and 
a Markov chain of order $m$ whose corresponding probabilistic Lyapunov exponent is equal to $\rhod(\mathcal A)$ (see Corollary~\ref{CoroFinitenessProp}). This, in turns, is equivalent to say that the finiteness property holds if and only if the set of maximizing measures contains the measure induced by some Markov chain of arbitrary order.

\medskip

\noindent {\bf Acknowledgements.} The authors are indebted with D.~Chafaï for helpful discussions. They are also grateful to the anonymous reviewers of a preceding version of the ma\-nus\-cript for providing helpful comments and pointing out relevant literature.

\section{Definitions, notations, and basic facts}
\label{SecDefinitions}

Throughout the paper, $d$ and $N$ belong to $\mathbb N$, which is used to denote the set of positive integers. If $a$ and $b$ are positive integers, $\llbracket a, b \rrbracket$ denotes the set of integers $j$ such that $a \leq j \leq b$. For $x \in \mathbb R$, $\ceil{x}$ denotes the smallest integer greater than or equal to $x$, and we extend this notation componentwise to vectors and matrices. We use $\norm{\cdot}$ to denote a norm in $\mathbb R^d$ as well as the corresponding induced norm on the space $\mathcal M_d(\mathbb R)$ of $d \times d$ matrices with real coefficients. We only consider in this paper norms in $\mathcal M_d(\mathbb R)$ obtained as induced norms from $\mathbb R^d$. An $N$-tuple $\mathcal A = (A_1, \dotsc, A_N) \in \mathcal M_d(\mathbb R)^N$ is said to be \emph{irreducible} if the only subspaces of $\mathbb R^d$ invariant under all the matrices $A_1, \dotsc, A_N$ are $\{0\}$ and $\mathbb R^d$.

\subsection{Deterministic joint spectral radius}

Let $\Sigma(\mathcal A)$ be the discrete-time switched system defined in \eqref{SwitchedSystem}. The \emph{deterministic joint spectral radius} $\rhod(\mathcal A)$ of $\Sigma(\mathcal A)$, introduced in \cite{RotaStrang1960}, is defined by
\begin{equation}
\label{eq:def:rhod}
\rhod(\mathcal A) = \limsup_{n \to \infty} \max_{(i_1, \dotsc, i_n) \in \llbracket 1, N\rrbracket^n} \norm*{A_{i_n} \dotsm A_{i_1}}^{1/n}.
\end{equation}
Since all norms in $\mathcal M_d(\mathbb R)$ induced by norms in $\mathbb R^d$ are submultiplicative and equivalent to each other, it immediately follows that $\rhod(\mathcal A)$ does not depend on a specific choice of such a norm and that
%Since all norms in $\mathcal M_d(\mathbb R)$ are equivalent, it immediately follows that $\rhod(\mathcal A)$ does not depend on the specific choice of $\norm{\cdot}$. Since $\norm{\cdot}$ is submultiplicative, one also has that
\[
\rhod(\mathcal A) = \lim_{n \to \infty} \max_{(i_1, \dotsc, i_n) \in \llbracket 1, N\rrbracket^n} \norm*{A_{i_n} \dotsm A_{i_1}}^{1/n} = \inf_{n \in \mathbb N} \max_{(i_1, \dotsc, i_n) \in \llbracket 1, N\rrbracket^n} \norm*{A_{i_n} \dotsm A_{i_1}}^{1/n}.
\]
%Moreover, the Joint Spectral Radius Theorem (see, e.g., \cite[Theorem 2.3]{Jungers2009Joint}) asserts that
%\[
%\rhod(\mathcal A) = \limsup_{n \to \infty} \max_{(i_1, \dotsc, i_n) \in \llbracket 1, N\rrbracket^n} \rho\left(A_{i_n} \dotsm A_{i_1}\right)^{1/n}.
%\]
Notice that, for every $n \in \mathbb N$ and $(i_1, \dotsc, i_n) \in \llbracket 1, N\rrbracket^n$, we have
\begin{equation}
\label{RhoDGeqRhoWord}
\rho\left(A_{i_n} \dotsm A_{i_1}\right)^{1/n} \leq \rhod(\mathcal A),
\end{equation}
where we use the definition of $\rhod(\mathcal A)$ and the fact that $\rho(M) = \rho(M^k)^{1/k} \leq \norm{M^k}^{1/k}$ for every square matrix $M$ and $k \in \mathbb N$.

\begin{defi}[Barabanov norm]
\label{DefiBarabanov}
Let $\mathcal A = (A_1, \dotsc, A_N)$ be an $N$-tuple of $d \times d$ matrices with real coefficients. A norm $\norm{\cdot}_{\mathrm B}$ on $\mathbb R^d$ is said to be a \emph{Barabanov norm} for $\mathcal A$ if the following two conditions hold.
\begin{enumerate}
\item\label{DefiBarabanov-Extremal} For every $i \in \llbracket 1, N\rrbracket$, $\norm{A_i}_{\mathrm B} \leq \rhod(\mathcal A)$.
\item For every $x \in \mathbb R^d$ and $k \in \mathbb N$, there exists $\sigma \in \mathfrak S$ such that $\norm{A_{\sigma(k)} \dotsm A_{\sigma(1)} x}_{\mathrm B} = \rhod(\mathcal A)^k \norm{x}_{\mathrm B}$.
\end{enumerate}
\end{defi}

%\begin{remk}
%For every $i \in \llbracket 1, N\rrbracket$, one has $\norm{A_i}_{\mathrm B} \leq \rhod(\mathcal A)$
%\end{remk}

The following basic result on Barabonov norms was proved in \cite{BarabanovDiscrete}.

\begin{prop}
\label{PropBarabanov}
Let $\mathcal A$ be an $N$-tuple of $d \times d$ matrices with real coefficients. If $\mathcal A$ is irreducible, then it admits a Barabanov norm.
\end{prop}

\subsection{Probabilistic joint spectral radius}

We now provide a probabilistic counterpart to $\rhod(\mathcal A)$. For that purpose, we collect some basic notions concerning transition matrices of Markov chains.

\begin{defi}
\label{DefiMarkov}
Let $P = (p_{ij})_{1 \leq i, j \leq N}$ be an $N \times N$ matrix with nonnegative coefficients.
\begin{enumerate}
\item $P$ is said to be \emph{stochastic} if, for every $i \in \llbracket 1, N\rrbracket$, $\sum_{j=1}^N p_{ij} = 1$.
\item\label{ItemStronglyConnected} $P$ is said to be \emph{strongly connected} if it is not similar via a permutation to an upper block triangular matrix.
\item For $k \in \mathbb N$ and $i_1, \dotsc, i_k \in \llbracket 1, N\rrbracket$, we say that $(i_1, \dotsc, i_k)$ is a \emph{$P$-word} if $p_{i_1 i_2} p_{i_2 i_3} \dotsm \allowbreak p_{i_{k-1} i_k} > 0$. The integer $k$ is called the \emph{length} of the $P$-word $(i_1, \dotsc, i_k)$. We say that $(i_1, \dotsc, i_k)$ is a \emph{$P$-cycle} if $p_{i_1 i_2} p_{i_2 i_3} \dotsm p_{i_{k-1} i_k} p_{i_k i_1} > 0$. The index $i_1$ is called the \emph{starting index} of the $P$-cycle $(i_1, \dotsc, i_k)$.

\item Let $\nu$ be a vector in $\mathbb R^N$ with nonnegative coefficients. We say that $(i_1, \dotsc, i_k)$ is a \emph{$(\nu, P)$-word} (respectively, \emph{$(\nu, P)$-cycle}) if it is a $P$-word (respectively, $P$-cycle) and $\nu_{i_1} > 0$.

\item If $P$ is stochastic, a row vector $\nu = (\nu_1, \dotsc, \nu_N) \in \mathbb R^N$ is said to be an \emph{invariant probability} for $P$ if $\nu_i \geq 0$ for every $i \in \llbracket 1, N\rrbracket$, $\sum_{i=1}^N \nu_i = 1$, and $\nu = \nu P$.
\end{enumerate}
\end{defi}

\begin{remk}
\label{RemkStronglyConnected}
In the context of discrete-time Markov chains in a finite state space with $N$ states, the \emph{transition matrix} is the stochastic matrix $P = (p_{ij})_{1 \leq i, j \leq N}$ where $p_{ij}$ represents the probability to switch from the state $i$ to the state $j$. Notice that $P$ is strongly connected if and only if its associated oriented graph $G_P$ is strongly connected. In the stochastic processes literature, the strong connectedness of $P$ is more often referred to as \emph{irreducibility}. We choose to stick with the former to avoid ambiguities with the homonymous notion for $N$-tuples of matrices. Notice also that the notions of strong connectedness, $P$-cycles, and $P$-words only depend on the adjacency matrix $\ceil{P}$ of the graph $G_P$, while $(\nu, P)$-cycles and $(\nu, P)$-words depend on $\ceil{P}$ and $\ceil{\nu}$.
\end{remk}

\begin{remk}
\label{RemkInvariantProbability}
Recall that, by the Perron--Frobenius Theorem, a stochastic matrix $P$ always admits an invariant probability, which is unique and has positive entries if $P$ is strongly connected. In the latter case, the definitions of $P$-word and $(\nu, P)$-word coincide, as well as those of $P$-cycle and $(\nu, P)$-cycle.
\end{remk}

We have the following classical decomposition result for stochastic matrices \cite[\S\S 1.2 and 4.2]{Seneta2006Nonnegative}.

\begin{prop}
\label{PropStochasticDecomposition}
Let $P \in \mathcal M_N(\mathbb R)$ be a stochastic matrix. Then, up to a permutation in the set of indices $\llbracket 1, N\rrbracket$, $P$ is given by
\begin{equation}
\label{DecomposeP}
P = \begin{pmatrix}
P_{1} & 0 & \cdots & 0 & 0 & 0 \\
0 & P_2 & 0 & \cdots & 0 & 0 \\
\vdots & 0 & \ddots & \ddots & \vdots & \vdots \\
0 & \vdots &\ddots & \ddots & 0 & 0 \\
0 & 0 & \cdots & 0 & P_{R} & 0 \\
* & * & \cdots & * & * & Q\\
\end{pmatrix},
\end{equation}
where $\rho(Q) < 1$ and, for $i \in \llbracket 1, R\rrbracket$, $P_i \in \mathcal M_{n_i}(\mathbb R)$ is a stochastic and strongly connected matrix.

Moreover, for $i \in \llbracket 1, R\rrbracket$, let $\nu^{[i]}$ be the unique invariant probability for $P_i$ and denote by the same symbol its canonical extension as a vector in $\mathbb R^N$ according to the decomposition \eqref{DecomposeP}. Then every invariant probability $\nu \in \mathbb R^N$ can be uniquely decomposed as
\begin{equation}
\label{DecomposeNu}
\nu = \sum_{i=1}^R \alpha_i \nu^{[i]},
\end{equation}
where $\alpha_1, \dotsc, \alpha_R \in [0, 1]$ and $\sum_{i=1}^R \alpha_i = 1$.
\end{prop}

The next lemma, useful in the proof of some of our results, uses the previous decomposition to obtain that any $(\nu, P)$-cycle has all its indices corresponding to a same diagonal block $P_i$ in \eqref{DecomposeP}.

\begin{lemm}
\label{LemmDecomposeCycle}
Let $P \in \mathcal M_N(\mathbb R)$ be a stochastic matrix decomposed according to Proposition \ref{PropStochasticDecomposition}. For $i \in \llbracket 1, R\rrbracket$, let
\[
\mathcal I_i = \left\llbracket 1 + \sum_{j=1}^{i-1} n_j, \sum_{j=1}^i n_j \right\rrbracket,
\]
i.e., $\mathcal I_i$ is the set of indices corresponding to the diagonal block $P_i$ in \eqref{DecomposeP}. Let $\nu$ be an invariant probability for $P$. Then, for every $(\nu, P)$-cycle $(i_1, \dotsc, i_n)$, there exists $j \in \llbracket 1, R\rrbracket$ such that $i_1, \dotsc, i_n$ are in $\mathcal I_j$.
\end{lemm}

\begin{proof}
Notice that, by \eqref{DecomposeNu}, $\nu_i = 0$ if $i \notin \bigcup_{j \in \llbracket 1, R\rrbracket}\mathcal I_j$. Hence, since $\nu_{i_1} > 0$, there exists $j \in \llbracket 1, R\rrbracket$ such that $i_1 \in \mathcal I_j$. Since $p_{i_1 i_2} > 0$, it follows by the block decomposition \eqref{DecomposeP} that $i_2 \in \mathcal I_j$. The conclusion follows by an immediate inductive argument.
\end{proof}

We also introduce the following notation.

\begin{defi}
Let $P$ be a stochastic matrix and $\mathcal A = (A_1, \dotsc, A_N)$ be an $N$-tuple of $d \times d$ matrices with real coefficients.
\begin{enumerate}
\item For every $P$-word $(i_1, \dotsc, i_k)$, we use $A(i_1, \dotsc, i_k)$ to denote the matrix product $A_{i_k} \dotsm \allowbreak A_{i_1}$. 
\item For every $s \in \llbracket 1, N\rrbracket$, let $C(P, s)$ be the matrix semigroup made of all matrix products associated with $P$-cycles with starting index $s$, i.e.,
\[
C(P, s) = \{A(i_1, \dotsc, i_k) \mid (i_1, \dotsc, i_k) \text{ is a $P$-cycle and } i_1 = s\}.
\]
We also set
\[
C(P) = \bigcup_{s \in \llbracket 1, N\rrbracket} C(P, s).
\]
\end{enumerate}
\end{defi}

We finally provide the definition of the probabilistic counterpart of $\rhod(\mathcal A)$ for $\Sigma(\mathcal A)$. Let $P = (p_{ij})_{1 \leq i, j \leq N}$ be a stochastic matrix, $\nu = (\nu_1, \dotsc, \nu_N)$ be an invariant probability for $P$, and $\mathcal A = (A_1, \dotsc, A_N)$ an $N$-tuple in $\mathcal M_d(\mathbb R)$. The \emph{probabilistic joint spectral radius} $\rhop(\nu, P, \mathcal A)$ is defined as
\begin{equation}
\label{EqDefiLambdaP}
\rhop(\nu, P, \mathcal A) = \limsup_{n \to \infty} \mathbb E_{(\nu, P)} \left[\norm*{A_{i_n} \dotsm A_{i_1}}^{1/n}\right],
\end{equation}
where
\begin{equation}
\label{Esperance}
\mathbb E_{(\nu, P)} \left[\norm*{A_{i_n} \dotsm A_{i_1}}^{1/n}\right] = \sum_{(i_1, \dotsc, i_n) \in \llbracket 1, N\rrbracket^n} \nu_{i_1} p_{i_1 i_2} \dotsm p_{i_{n-1} i_n} \norm*{A_{i_n} \dotsm A_{i_1}}^{1/n}.
\end{equation}
As in the deterministic case, $\rhop(\nu, P, \mathcal A)$ does not depend on the specific choice of the norm $\norm{\cdot}$ and we have
\begin{equation}
\label{RhoPLimInf}
\rhop(\nu, P, \mathcal A) = \lim_{n \to \infty} \mathbb E_{(\nu, P)} \left[\norm*{A_{i_n} \dotsm A_{i_1}}^{1/n}\right] = \inf_{n \in \mathbb N} \mathbb E_{(\nu, P)} \left[\norm*{A_{i_n} \dotsm A_{i_1}}^{1/n}\right].
\end{equation}

\begin{remk}
\label{RemkSumvPWord}
The expectation in \eqref{EqDefiLambdaP} is taken with respect to the random variable $(i_1, \dotsc,\allowbreak i_n) \in \llbracket 1, N\rrbracket^n$. The definition of probabilistic joint spectral radius provided here is a particular instance of a more general and comprehensive formulation based on symbolic dynamics; see, for instance, \cite{Dai2011Periodically, Dai2008Almost, Morris2013Mather}. Notice also that it follows from the definition of $(\nu, P)$-word that the summation in \eqref{Esperance} can be restricted to $(\nu, P)$-words of length $n$.
\end{remk}

\color{blue}
\begin{remk}
\label{remk:Lyap-exp}
When dealing with probabilistic switching phenomena in discrete time, several works, such as \cite{Morris2013Mather, Gerencser2008Stability, Furstenberg1960Products, Dai2008Almost, Crisanti1993Products, Arnold1998Random}, deal with the \emph{probabilistic Lyapunov exponent} $\lambda_{\mathrm p}(\nu, P, \mathcal A)$ defined by
\[
\lambda_{\mathrm p}(\nu, P, \mathcal A) = \limsup_{n \to \infty} \frac{1}{n} \mathbb E_{(\nu, P)} \left[\log \norm*{A_{i_n} \dotsm A_{i_1}}\right].
\]
Our choice to use $\rhop(\nu, P, \mathcal A)$ instead is motivated by the fact that the main goal of our paper is to compare the probabilistic behavior of \eqref{SwitchedSystem} with the worst deterministic behavior provided by the classical joint spectral radius $\rhod(\mathcal A)$, whose definition in discrete-time \eqref{eq:def:rhod} does not involve taking the logarithm of the norm of the matrix product. Working with $\rhop(\nu, P, \mathcal A)$ also has the additional advantage of being able to handle the case $\rhop(\nu, P, \mathcal A) = 0$ without dealing with the singularity of the logarithm at $0$.

Clearly, by Jensen's inequality, we have $e^{\lambda_{\mathrm p}(\nu, P, \mathcal A)} \leq \rhop(\nu, P, \mathcal A)$, but this inequality may be strict in some cases. Indeed, for $d = 1$, $N = 2$, $P = \id_2$, $\nu = (\frac{1}{2}, \frac{1}{2})$, and $\mathcal A = (A_1, A_2) \in \mathcal M_1(\mathbb R)^2 \simeq \mathbb R^2$, we easily compute that $\rhop(\nu, P, \mathcal A) = \frac{1}{2}(A_1 + A_2)$ and $e^{\lambda_{\mathrm p}}(\nu, P, \mathcal A) = \sqrt{A_1 A_2}$.

We do have equality between $e^{\lambda_{\mathrm p}(\nu, P, \mathcal A)}$ and $\rhop(\nu, P, \mathcal A)$, however, under the assumption that $(\nu, P)$ defines an \emph{ergodic} Markov chain, i.e., $\nu = \nu^{[i]}$ for some $i \in \llbracket 1, R\rrbracket$ in the decompositions \eqref{DecomposeP} and \eqref{DecomposeNu} in Proposition \ref{PropStochasticDecomposition}. Indeed, in this case, the main result of \cite{Furstenberg1960Products} implies that
\[
\lambda_{\mathrm p}(\nu, P, \mathcal A) = \lim_{n \to \infty} \frac{1}{n} \log \norm*{A_{\sigma(n)} \dotsm A_{\sigma(1)}} \qquad \text{ for $\mathbb P_{(\nu, P)}$-almost every $\sigma \in \mathfrak S$},
\]
where $\mathbb P_{(\nu, P)}$ denotes the probability measure on $\mathfrak S$ associated canonically with the transition matrix $P$ and the invariant probability $\nu$. Using this fact, one deduces that $e^{\lambda_{\mathrm p}(\nu, P, \mathcal A)} = \rhop(\nu, P, \mathcal A)$, and, in addition, we also have the equality
\begin{equation}
\label{eq:F-K}
\rhop(\nu, P, \mathcal A) = \lim_{n \to \infty} \norm*{A_{\sigma(n)} \dotsm A_{\sigma(1)}}^{1/n} \qquad \text{ for $\mathbb P_{(\nu, P)}$-almost every $\sigma \in \mathfrak S$}.
\end{equation}
Notice that, in particular, $(\nu, P)$ is ergodic when $P$ is strongly connected and $\nu$ is its unique invariant measure.
\end{remk}
\color{black}

\begin{remk}
The deterministic joint spectral radius $\rhod(\mathcal A)$ provides the worst asymptotic behavior of $\Sigma(\mathcal A)$ with respect to $\sigma \in \mathfrak S$. By introducing the probability measure $\mathbb P_{(\nu, P)}$ on $\mathfrak S$ associated canonically with the transition matrix $P$ and the invariant probability $\nu$, the quantity $\rhop(\nu, P, \mathcal A)$ defined in \eqref{EqDefiLambdaP} can be interpreted as an asymptotic behavior averaged by $\mathbb P_{(\nu, P)}$. \textcolor{blue}{When $(\nu, P)$ is ergodic, thanks to \eqref{eq:F-K}, we have the stronger interpretation of $\rhop(\nu, P, \mathcal A)$ as the $\mathbb P_{(\nu, P)}$-almost sure asymptotic behavior of $\Sigma(\mathcal A)$.}% More precisely, if, in the decompositions \eqref{DecomposeP} and \eqref{DecomposeNu} in Proposition \ref{PropStochasticDecomposition}, $\nu = \nu^{[i]}$ for some $i \in \llbracket 1, R\rrbracket$, then the main result of \cite{Furstenberg1960Products} implies that, for $\mathbb P_{(\nu, P)}$-almost every $\sigma \in \mathfrak S$,
%\[
%\rhop(\nu, P, \mathcal A) = \lim_{n \to \infty} \norm*{A_{\sigma(n)} \dotsm A_{\sigma(1)}}^{1/n}.
%\]
%Notice that the above assumption is satisfied when $P$ is strongly connected.
\end{remk}

It is immediate to see that, for every $(\nu, P, \mathcal A)$ as above, we have $\rhop(\nu, P, \mathcal A) \leq \rhod(\mathcal A)$, and then
\begin{equation}
\label{Inequality}
\rhop(\nu, P, \mathcal A) \leq \sup_{\nu^\prime} \rhop(\nu^\prime, P, \mathcal A) \leq \sup_{(\nu^\prime, P^\prime)} \rhop(\nu^\prime, P^\prime, \mathcal A) \leq \rhod(\mathcal A),
\end{equation}
where the first supremum is taken over all invariant probabilities $\nu^\prime$ for $P$ and the second one over the pairs $(\nu^\prime, P^\prime)$ made of an $N \times N$ stochastic matrix $P^\prime$ and an invariant probability $\nu^\prime$ for $P^\prime$. We find it useful to introduce the notation
\begin{equation}
\label{EqDefiRhoPPA}
\rhop(P, \mathcal A) = \sup_{\nu^\prime} \rhop(\nu^\prime, P, \mathcal A), \qquad \rhop(\mathcal A) = \sup_{(\nu^\prime, P^\prime)} \rhop(\nu^\prime, P^\prime, \mathcal A).
\end{equation}

\begin{remk}
\label{RemkSupAttained}
It follows from \eqref{RhoPLimInf} that $(\nu^\prime, P^\prime) \mapsto \rhop(\nu^\prime, P^\prime, \mathcal A)$ is upper semicontinuous. Moreover, the set of pairs $(\nu^\prime, P^\prime)$ consisting of an $N \times N$ stochastic matrix $P^\prime$ and an invariant probability $\nu^\prime$ for $P^\prime$ is compact. As a consequence, the suprema in \eqref{EqDefiRhoPPA} can be replaced by maxima.
\end{remk}

\section{Equality between deterministic and probabilistic joint spectral radii}
\label{SecResults}

\subsection{Equality between \texorpdfstring{$\rhod(\mathcal A)$}{rho d(A)} and \texorpdfstring{$\rhop(\nu, P, \mathcal A)$}{rho p(nu, P, A)}}
\label{SecNuP}

The goal of this section is to prove the following result characterizing equality between $\rhod(\mathcal A)$ and $\rhop(\nu, P, \mathcal A)$.

\begin{theo}
\label{MainTheoFixedP}
Let $P \in \mathcal M_N(\mathbb R)$ be a stochastic matrix, $\nu \in \mathbb R^N$ be an invariant probability measure for $P$, and $\mathcal A = (A_1,\allowbreak \dotsc,\allowbreak A_N) \in \mathcal M_d(\mathbb R)^N$. Then the following statements are equivalent:
\begin{enumerate}
\item\label{MainTheoFixedP-Equality} $\rhod(\mathcal A) = \rhop(\nu, P, \mathcal A)$.
\item\label{MainTheoFixedP-Spr} $\rho(A_{i_k} \dotsm A_{i_1})^{1/k} = \rhod (\mathcal A)$ for every $(\nu, P)$-cycle $(i_1, \dotsc, i_k)$.
\end{enumerate}
\end{theo}

The fact that \ref{MainTheoFixedP-Equality} implies \ref{MainTheoFixedP-Spr} follows from the results in \cite{Morris2013Mather}, as detailed in the following lemma.

\begin{lemm}
\label{LemmMorris}
Let $P \in \mathcal M_N(\mathbb R)$ be a stochastic matrix, $\nu \in \mathbb R^N$ be an invariant probability measure for $P$, and $\mathcal A = (A_1,\allowbreak \dotsc,\allowbreak A_N) \in \mathcal M_d(\mathbb R)^N$. If $\rhod(\mathcal A) = \rhop(\nu, P, \mathcal A)$, then $\rho(A_{i_k} \dotsm A_{i_1})^{1/k} = \rhod (\mathcal A)$ for every $(\nu, P)$-cycle $(i_1, \dotsc, i_k)$.
\end{lemm}

\begin{proof}
If $\rhod(\mathcal A) = 0$, the result follows trivially from \eqref{RhoDGeqRhoWord}. We then assume $\rhod(\mathcal A) > 0$,
\color{blue}
we decompose $P$ and $\nu$ according to Proposition \ref{PropStochasticDecomposition}, and we use in the sequel the same notations as in its statement. We also let $\mathcal I_1, \dotsc, \mathcal I_R$ be defined as in the statement of Lemma~\ref{LemmDecomposeCycle}. Thanks to \eqref{DecomposeNu}, \eqref{EqDefiLambdaP}, and \eqref{Esperance}, we have
\begin{equation}
\label{LambdaPConvexInNu-0}
\rhop(\nu, P, \mathcal A) = \sum_{j=1}^R \alpha_j \rhop(\nu^{[j]}, P, \mathcal A).
\end{equation}
By \eqref{Inequality}, we have $\rhop(\nu^{[j]}, P, \mathcal A) \leq \rhod(\mathcal A)$ for every $j \in \llbracket 1, R\rrbracket$ and, since $\alpha_j \in [0, 1]$ for every $j \in \llbracket 1, R\rrbracket$ and $\sum_{j=1}^R \alpha_j = 1$, we deduce from \eqref{LambdaPConvexInNu-0} and the equality $\rhod(\mathcal A) = \rhop(\nu, P, \mathcal A)$ that $\rhop(\nu^{[j]}, P, \mathcal A) = \rhod(\mathcal A)$ for every $j \in \llbracket 1, R\rrbracket$ such that $\alpha_j > 0$. 

Let $(i_1, \dotsc, i_k)$ be a $(\nu, P)$-cycle and note that, by Lemma~\ref{LemmDecomposeCycle}, there exists $r \in \llbracket 1, R\rrbracket$ such that $i_1, \dotsc, i_k$ are in $\mathcal I_r$, and thus, in particular, $(i_1, \dotsc, i_k)$ is also a $(\nu^{[r]}, P)$-cycle. Moreover, such a $r$ necessarily satisfies $\alpha_r > 0$.
\color{black}
Consider the $k$-periodic switching signal $\sigma \in \mathfrak S$ corresponding to $(i_1, \dotsc, i_k)$, defined by $\sigma(j + \ell k) = i_j$ for all integers $j \in \llbracket 1, k\rrbracket$ and $\ell \geq 0$. Endow $\mathfrak S$ with its usual product topology and denote by \textcolor{blue}{$\mathbb P_r$} the Borel probability measure on $\mathfrak S$ corresponding to the Markov chain defined by \textcolor{blue}{$(\nu^{[r]}, P)$}. Note that, since $(i_1, \dotsc, i_k)$ is a \textcolor{blue}{$(\nu^{[r]}, P)$}-cycle, for every $n \in \mathbb N$ the set $\{\widetilde\sigma \in \mathfrak S \mid \widetilde\sigma(i) = \sigma(i) \text{ for every } i \in \llbracket 1, n\rrbracket\}$ has positive \textcolor{blue}{$\mathbb P_r$} measure, and thus $\sigma$ is in the support of \textcolor{blue}{$\mathbb P_r$}. Moreover, \textcolor{blue}{using also Remark~\ref{remk:Lyap-exp}, we have $\rhod(\mathcal A) = \rhop(\nu^{[r]}, P, \mathcal A) = e^{\lambda_{\mathrm p}(\nu^{[r]}, P, A)}$, and then $\mathbb P_r$} is a maximizing measure of $\mathcal A$ in the sense of \cite{Morris2013Mather}, where the set of maximizing measures of $\mathcal A$ is defined as the set of all Borel probability measures on $\mathfrak S$ invariant under the usual time shift and such that the corresponding probabilistic \textcolor{blue}{Lyapunov exponent} coincides with \textcolor{blue}{$\log\rhod(\mathcal A)$}. Hence $\sigma$ belongs to the Mather set of $\mathcal A$ (see \cite[Theorem~2.3]{Morris2013Mather}, where the Mather set of $\mathcal A$ is defined as the union of the supports of all maximizing measures of $\mathcal A$), and thus, by \cite[Theorem~2.3(3)]{Morris2013Mather}, we get
\begin{equation}
\label{LimsupMorris}
\limsup_{n \to \infty} \rhod(\mathcal A)^{-n} \rho(A_{\sigma(n)} \dotsm A_{\sigma(1)}) = 1.
\end{equation}
Set $M = \rhod(\mathcal A)^{-k} A_{i_k} \dotsm A_{i_1}$. By \eqref{RhoDGeqRhoWord}, we have that $\rho(M) \leq 1$. For every $n \geq 1$, there exist integers $\ell \geq 0$ and $j \in \llbracket 0, k-1\rrbracket$ such that $n = j + \ell k$. Since $\sigma$ is $k$-periodic, we have that
\[
\rhod(\mathcal A)^{-n} \rho(A_{\sigma(n)} \dotsm A_{\sigma(1)}) = \rhod(\mathcal A)^{-j} \rho(A_{i_j} \dotsm A_{i_1} M^\ell).
\]
If $\rho(M) < 1$, then the right-hand side of the above inequality tends to $0$ as $\ell \to \infty$, contradicting \eqref{LimsupMorris}. Hence, we have necessarily $\rho(M) = 1$.
\end{proof}

The proof that \ref{MainTheoFixedP-Spr} implies \ref{MainTheoFixedP-Equality} in Theorem \ref{MainTheoFixedP} is decomposed in three steps. We first establish the result under the extra assumptions that $\mathcal A$ is irreducible and $P$ is strongly connected (Lemma~\ref{Lemm-PSC-AIrreducible}). We then obtain the conclusion under the sole additional assumption that $P$ is strongly connected (Lemma~\ref{Lemm-PSC}). Finally, we consider the general case in the third step.

%We start with the following proposition.

\begin{lemm}
\label{Lemm-PSC-AIrreducible}
Let $P \in \mathcal M_d(\mathbb R)$ be a stochastic strongly connected matrix, $\mathcal A = (A_1, \dotsc,\allowbreak A_N) \allowbreak \in \mathcal M_d(\mathbb R)^N$ be irreducible, and $\norm{\cdot}_{\mathrm B}$ be a Barabanov norm for $\mathcal A$. Then the following statements are equivalent:
\begin{enumerate}
\item\label{Lemm-PSC-AIrreducible-Equality} $\rhod(\mathcal A) = \rhop(P, \mathcal A)$.
\item\label{Lemm-PSC-AIrreducible-Spr} $\rho(A_{i_k} \dotsm A_{i_1})^{1/k} = \rhod(\mathcal A)$ for every $P$-cycle $(i_1, \dotsc, i_k)$.
\item\label{Lemm-PSC-AIrreducible-Barabanov} $\norm*{A_{i_k} \dotsm A_{i_1}}_{\mathrm B}^{1/k} = \rhod(\mathcal A)$ for every $P$-word $(i_1, \dotsc, i_k)$.
\end{enumerate}
\end{lemm}

\begin{proof}
The fact that \ref{Lemm-PSC-AIrreducible-Equality} implies \ref{Lemm-PSC-AIrreducible-Spr} is a particular case of Lemma~\ref{LemmMorris}. Moreover, it is immediate that \ref{Lemm-PSC-AIrreducible-Barabanov} implies \ref{Lemm-PSC-AIrreducible-Equality} thanks to \eqref{EqDefiLambdaP}, \eqref{Esperance}, and Remark \ref{RemkSumvPWord}. We are then left to prove that \ref{Lemm-PSC-AIrreducible-Spr} implies \ref{Lemm-PSC-AIrreducible-Barabanov}.

 %We prove the converse by contrapositive. Notice that, by definition of Barabanov norm, one has $\norm*{A_{i_n} \dotsm A_{i_1}}_{\mathrm B} \leq \rhod(\mathcal A)^n$ for every $n \in \mathbb N$ and every $P$-word $(i_1, \dotsc, i_n)$. Assuming that \ref{Lemm-PSC-AIrreducible-Barabanov} does not hold, there exist $n \in \mathbb N$ and a $P$-word $(i_1, \dotsc, i_{n})$ such that
%\[
%\norm*{A_{i_{n}} \dotsm A_{i_1}}_{\mathrm B}^{1/n} < \rhod(\mathcal A).
%\]
%Then, by \eqref{Esperance} and Remark \ref{RemkInvariantProbability},
%\[
%\mathbb E_{(\nu, P)} \left[\norm*{A_{i_{n}} \dotsm A_{i_1}}_{\mathrm B}^{1/{n}}\right] < \rhod(\mathcal A),
%\]
%which implies that $\rhop(P, \mathcal A) < \rhod(\mathcal A)$ according to \eqref{RhoPLimInf}. This concludes the proof of the equivalence between \ref{Lemm-PSC-AIrreducible-Equality} and \ref{Lemm-PSC-AIrreducible-Barabanov}.

%We now establish the equivalence between \ref{Lemm-PSC-AIrreducible-Barabanov} and \ref{Lemm-PSC-AIrreducible-Spr}. Assume first that \ref{Lemm-PSC-AIrreducible-Barabanov} holds. If $(i_1, \dotsc, i_k)$ is a $P$-cycle, we define $\sigma \in \mathfrak S$ by setting $\sigma(j) = i_\ell$ when $j \equiv \ell \mod k$. Then $(\sigma(1), \dotsc, \sigma(rk))$ is a $P$-word of length $rk$ for every $r \in \mathbb N$ and, by hypothesis, one has
%\[
%\norm*{(A_{i_k} \dotsm A_{i_1})^r}_{\mathrm B}^{1/{rk}} = \rhod(\mathcal A).
%\]
%Letting $r \to \infty$ and using Gelfand's formula, one concludes that
%\[
%\rho(A_{i_k} \dotsm A_{i_1})^{1/k} = \rhod(\mathcal A),
%\]
%which yields \ref{Lemm-PSC-AIrreducible-Spr}.

Assume that \ref{Lemm-PSC-AIrreducible-Spr} holds. Fix a $P$-word $(i_1, \dotsc, i_k)$. Since $P$ is strongly connected, there exist $r \in \mathbb N$ and $i_{k+1}, \dotsc, i_r \in \llbracket 1, N\rrbracket$ (obtained by connecting $i_k$ to $i_1$) such that $(i_1, \dotsc, i_r)$ is a $P$-cycle. Then, by \ref{Lemm-PSC-AIrreducible-Spr},
\[
\rho(A_{i_r} \dotsm A_{i_1}) = \rhod(\mathcal A)^r.
\]
Since the spectral radius is a lower bound for any induced norm of a matrix, we obtain that
%\[
%\norm*{A_{i_r} \dotsm A_{i_1}}_{\mathrm B} \geq \rhod(\mathcal A)^r.
%\]
%Then
\[
\rhod(\mathcal A)^r \leq \norm*{A_{i_r} \dotsm A_{i_1}}_{\mathrm B} \leq \norm*{A_{i_r} \dotsm A_{i_{k+1}}}_{\mathrm B} \norm*{A_{i_k} \dotsm A_{i_1}}_{\mathrm B}.
\]
Using the fact that $\norm{\cdot}_{\mathrm B}$ is a Barabanov norm, we also have that
\[
\norm*{A_{i_r} \dotsm A_{i_{k+1}}}_{\mathrm B} \norm*{A_{i_k} \dotsm A_{i_1}}_{\mathrm B} \leq \rhod(\mathcal A)^{r - k} \rhod(\mathcal A)^k = \rhod(\mathcal A)^r.
\]
By combining the previous inequalities, it follows that $\norm*{A_{i_k} \dotsm A_{i_1}}_{\mathrm B} = \rhod(\mathcal A)^k$.
\end{proof}

\begin{remk}
The proof of Lemma~\ref{Lemm-PSC-AIrreducible} only uses that $\norm{\cdot}_{\mathrm B}$ is an extremal norm, i.e., it satisfies \ref{DefiBarabanov-Extremal} in Definition \ref{DefiBarabanov}. The irreducibility assumption on $\mathcal A$ could then be replaced by its nondefectiveness (we refer the reader to \cite[Section 2.1.2]{Jungers2009Joint} for details). However, we prefer to state Lemma~\ref{Lemm-PSC-AIrreducible} in terms of irreducibility since this condition is easier to handle: it can be checked more directly and, up to a linear change of coordinates, a reducible $\mathcal A$ can be put into block-triangular form with irreducible diagonal blocks. This block decomposition is a key argument in the proof of Lemma~\ref{Lemm-PSC}.
\end{remk}

We now consider the case where $\mathcal A$ is not necessarily irreducible. Here, a Barabanov norm for $\mathcal A$ in general does not exist, and hence item \ref{Lemm-PSC-AIrreducible-Barabanov} from Lemma~\ref{Lemm-PSC-AIrreducible} cannot be expected.

\begin{lemm}
\label{Lemm-PSC}
Let $P \in \mathcal M_N(\mathbb R)$ be a stochastic strongly connected matrix and $\mathcal A = (A_1,\allowbreak \dotsc,\allowbreak A_N) \in \mathcal M_d(\mathbb R)^N$. Then the following statements are equivalent:
\begin{enumerate}
\item\label{Lemm-PSC-Equality} $\rhod(\mathcal A) = \rhop(P, \mathcal A)$.
\item\label{Lemm-PSC-Spr} $\rho(A_{i_k} \dotsm A_{i_1})^{1/k} = \rhod(\mathcal A)$ for every $P$-cycle $(i_1, \dotsc, i_k)$.
\end{enumerate}
\end{lemm}

\begin{proof}
Before giving the core of the argument, we start with a set of remarks. First, up to a linear change of coordinates, $A_1, \dotsc, A_N$ can be presented in block-tri\-an\-gu\-lar form as
\begin{equation*}
%\label{DecomposeAj}
A_j = \begin{pmatrix}
A_{j}^{(1)} & \ast & \ast & \cdots & \ast \\
0 & A_{j}^{(2)} & \ast & \cdots & \ast \\
0 & 0 & A_{j}^{(3)} & \cdots & \ast \\
\vdots & \vdots & \vdots & \ddots & \vdots \\
0 & 0 & 0 & \cdots & A_{j}^{(R)}
\end{pmatrix}, \qquad j \in \llbracket 1, N\rrbracket,
\end{equation*}
with $\mathcal A^{(r)} = (A_1^{(r)}, \dotsc, A_N^{(r)})$ irreducible for every $r \in \llbracket 1, R\rrbracket$. Remark that, on the one hand, according to \cite[Proposition 1.5]{Jungers2009Joint}, we have $\rhod(\mathcal A) = \max_{r \in \llbracket 1, R\rrbracket} \rhod(\mathcal A^{(r)})$ and, on the other hand, it follows from \cite[Theorem 1.1]{Gerencser2008Stability} and the strong connectedness of $P$ that $\rhop(P, \mathcal A) = \max_{r \in \llbracket 1, R\rrbracket} \allowbreak \rhop(\allowbreak P,\allowbreak \mathcal A^{(r)})$. Moreover, for every $P$-cycle $(i_1, \dotsc, i_k)$, we have
\begin{equation}
\label{RhoCycleMaxDiagonal}
\rhod(\mathcal A) \geq \rho(A_{i_k} \dotsm A_{i_1})^{1/k} = \max_{r \in \llbracket 1, R\rrbracket} \rho\left(A_{i_k}^{(r)} \dotsm A_{i_1}^{(r)}\right)^{1/k},
\end{equation}
where the inequality comes from \eqref{RhoDGeqRhoWord} and the equality results from the simple fact that the spectral radius of a block-triangular matrix is equal to the maximum of the spectral radii over the diagonal blocks.

Since \ref{Lemm-PSC-Equality} implies \ref{Lemm-PSC-Spr} by Lemma~\ref{LemmMorris}, we are left to prove the converse implication.
%
%We start by proving that \ref{Lemm-PSC-Equality} implies \ref{Lemm-PSC-Spr}. Let $\overline r \in \llbracket 1, N\rrbracket$ be such that $\rhop(P, \mathcal A^{(\overline r)}) = \max_{r \in \llbracket 1, N\rrbracket}\rhop(P, \mathcal A^{(r)}) = \rhop(P, \mathcal A)$. Then, by \ref{Lemm-PSC-Equality} and the previous remarks, one has
%\[\rhop(P, \mathcal A^{(\overline r)}) = \rhod(\mathcal A) = \max_{r \in \llbracket 1, N\rrbracket} \rhod(\mathcal A^{(r)}) \geq \rhod(\mathcal A^{(\overline r)}),\]
%and one obtains from \eqref{Inequality} that $\rhop(P, \mathcal A^{(\overline r)}) = \rhod(\mathcal A^{(\overline r)}) = \rhod(\mathcal A)$. By Lemma~\ref{Lemm-PSC-AIrreducible}, for every $P$-cycle $(i_1, \dotsc, i_k)$, one has, using \eqref{RhoCycleMaxDiagonal},
%\[
%\rhod(\mathcal A) = \rhod(\mathcal A^{(\overline r)}) = \rho\left(A_{i_k}^{(\overline r)} \dotsm A_{i_1}^{(\overline r)}\right)^{1/k} \leq \rho\left(A_{i_k} \dotsm A_{i_1}\right)^{1/k} \leq \rhod(\mathcal A).
%\]
%Then $\rho\left(A_{i_k} \dotsm A_{i_1}\right)^{1/k} = \rhod(\mathcal A)$.
%
Assume that \ref{Lemm-PSC-Spr} holds true. Then \ref{Lemm-PSC-Equality} holds trivially if $\rhod(\mathcal A) = 0$. Otherwise, we can assume, with no loss of generality, that $\rhod(\mathcal A) = 1$ up to replacing $\mathcal A$ by $\rhod(\mathcal A)^{-1} \mathcal A$. By assumption and \eqref{RhoCycleMaxDiagonal}, for every $P$-cycle $(i_1, \dotsc, i_k)$, there exists $r \in \llbracket 1, R\rrbracket$ such that $\rho\left(A_{i_k}^{(r)} \dotsm A_{i_1}^{(r)}\right) = 1$.

We claim that $r$ can be chosen independently of the $P$-cycle. We argue by contradiction, i.e., we assume that, for every $r \in \llbracket 1, R\rrbracket$, there exists a $P$-cycle $i^r = (i_1^r, \dotsc, i_{\ell_r}^r)$ such that $\rho(A^{(r)}(i^r)) < 1$. Let $j^r = (j_1^r, \dotsc, j_{k_r}^r)$ be a $P$-word such that $j_1^r = i_1^r$ and $p_{j_{k_r}^r i_1^{r+1}} > 0$ (with the convention that $i_1^{R+1} = i_1^1$). Then, for every $n \in \mathbb N$,
\[
A(j^R) A(i^R)^n \dotsm A(j^2) A(i^2)^n A(j^1) A(i^1)^n \in C(P).
\]
For every $n$, we apply \ref{Lemm-PSC-Spr} to the above product, and we deduce from \eqref{RhoCycleMaxDiagonal} that there exists $r_n \in \llbracket 1, R\rrbracket$ such that
\begin{multline*}
\rho\left(A^{(r_n)}(j^R) A^{(r_n)}(i^R)^{n} \dotsm A^{(r_n)}(j^2) A^{(r_n)}(i^2)^{n} A^{(r_n)}(j^1) A^{(r_n)}(i^1)^{n}\right) \\
 = \rho\left(A(j^R) A(i^R)^{n} \dotsm A(j^2) A(i^2)^{n} A(j^1) A(i^1)^{n}\right) = 1.
\end{multline*}
Let $(n_q)_{q \in \mathbb N}$ be an increasing sequence  such that there exists $\overline r \in \llbracket 1, R\rrbracket$ with $r_{n_q} = \overline r$ for every $q \in \mathbb N$. 
%\begin{align*}
%& \rho\left(A^{(r)}(j^R) A^{(r)}(i^R)^{n_q} \dotsm A^{(r)}(j^2) A^{(r)}(i^2)^{n_q} A^{(r)}(j^1) A^{(r)}(i^1)^{n_q}\right) \\
%{} = {} & \rho\left(A(j^R) A(i^R)^{n_q} \dotsm A(j^2) A(i^2)^{n_q} A(j^1) A(i^1)^{n_q}\right) = 1.
%\end{align*}
Since $\mathcal A^{(\overline r)}$ is irreducible, there exists a Barabanov norm $\norm{\cdot}_{\overline r}$ for $\mathcal A^{(\overline r)}$. Then, for every $q \in \mathbb N$, we have
\begin{align*}
1 & = \rho\left(A^{(\overline r)}(j^R) A^{(\overline r)}(i^R)^{n_q} \dotsm A^{(\overline r)}(j^2) A^{(\overline r)}(i^2)^{n_q} A^{(\overline r)}(j^1) A^{(\overline r)}(i^1)^{n_q}\right) \\
& \leq \norm*{A^{(\overline r)}(j^R) A^{(\overline r)}(i^R)^{n_q} \dotsm A^{(\overline r)}(j^2) A^{(\overline r)}(i^2)^{n_q} A^{(\overline r)}(j^1) A^{(\overline r)}(i^1)^{n_q}}_{\overline r} \\
& \leq \norm*{A^{(\overline r)}(i^{\overline r})^{n_q}}_{\overline r},
\end{align*}
where the last inequality follows from the fact that $\norm{\cdot}_{\overline r}$ is a Barabanov norm. Since $\rho(A^{(\overline r)}(i^{\overline r})) < 1$, we have that $\norm*{A^{(\overline r)}(i^{\overline r})^{n_q}}_{\overline r} \xrightarrow[q \to \infty]{} 0$, hence the contradiction.

We thus have proved that there exists $\overline r \in \llbracket 1, R\rrbracket$ such that, for every $P$-cycle $(i_1, \dotsc, i_k)$,
\[
\rho\left(A_{i_k}^{(\overline r)} \dotsm A_{i_1}^{(\overline r)}\right) = 1 = \rhod(\mathcal A).
\]
On the other hand, by \eqref{RhoDGeqRhoWord}, we have $\rho\left(A_{i_k}^{(\overline r)} \dotsm A_{i_1}^{(\overline r)}\right) \leq \rhod(\mathcal A^{(\overline r)})$. Since $\rhod(\mathcal A^{(\overline r)}) \leq \rhod(\mathcal A)$, we deduce that
\[
\rho\left(A_{i_k}^{(\overline r)} \dotsm A_{i_1}^{(\overline r)}\right) = \rhod(\mathcal A^{(\overline r)}) = \rhod(\mathcal A)
\]
for every $P$-cycle $(i_1, \dotsc, i_k)$. Then, using Lemma~\ref{Lemm-PSC-AIrreducible}, we obtain that 
\[\rhop(P, \mathcal A) \geq \rhop(P, \mathcal A^{(\overline r)}) = \rhod(\mathcal A^{(\overline r)}) = \rhod(\mathcal A),\]
and then \ref{Lemm-PSC-Equality} holds thanks to \eqref{Inequality}.
\end{proof}

We can conclude now the proof of Theorem \ref{MainTheoFixedP}.

\begin{proof}[Proof of Theorem \ref{MainTheoFixedP}]
Recall that, thanks to Lemma~\ref{LemmMorris}, we are only left to prove that \ref{MainTheoFixedP-Spr} implies \ref{MainTheoFixedP-Equality}. We first decompose $P$ and $\nu$ according to Proposition \ref{PropStochasticDecomposition} and use in the sequel the same notations as in its statement. Thanks to \eqref{EqDefiLambdaP} and \eqref{Esperance}, we have
\begin{equation}
\label{LambdaPConvexInNu}
\rhop(\nu, P, \mathcal A) = \sum_{j=1}^R \alpha_j \rhop(\nu^{[j]}, P, \mathcal A).
\end{equation}
For $j \in \llbracket 1, R\rrbracket$, let $\mathcal A^{[j]}$ be the ordered $n_j$-tuple made of the matrices $A_\ell$ such that $\nu^{[j]}_\ell > 0$. Notice that $\rhop(\nu^{[j]}, P_j, \mathcal A^{[j]}) = \rhop(\nu^{[j]}, P, \mathcal A)$ for every $j \in \llbracket 1, R\rrbracket$. Using \eqref{Inequality} and the fact that $\mathcal A^{[j]}$ is made of matrices from $\mathcal A$, we obtain that, for every $j \in \llbracket 1, R\rrbracket$,
\begin{equation}
\label{InegLambdaPNuJ}
\rhop(\nu^{[j]}, P_j, \mathcal A^{[j]}) \leq \rhod(\mathcal A^{[j]}) \leq \rhod(\mathcal A).
\end{equation}

%We now prove that \ref{MainTheoFixedP-Equality} implies \ref{MainTheoFixedP-Spr}. Using \eqref{LambdaPConvexInNu}, \eqref{InegLambdaPNuJ}, and \ref{MainTheoFixedP-Equality}, one deduces that
%\[\rhop(\nu^{[j]}, P_j, \mathcal A^{[j]}) = \rhod(\mathcal A^{[j]}) = \rhod(\mathcal A) \qquad \text{for every }j \in \llbracket 1, R\rrbracket \text{ such that } \alpha_j > 0.\]
%Let $(i_1, \dotsc, i_k)$ be a $(\nu, P)$-cycle. Then, letting $\mathcal I_1, \dotsc, \mathcal I_R$ be as in the statement of Lem\-ma~\ref{LemmDecomposeCycle}, one obtains that $i_1, \dotsc, i_k$ are in $\mathcal I_\ell$ for some $\ell \in \llbracket 1, R\rrbracket$. Moreover, since $\nu_{i_1} > 0$, one obtains that $\alpha_\ell > 0$. Hence, since $\rhop(\nu^{[\ell]}, P_\ell, \mathcal A^{[\ell]}) = \rhod(\mathcal A^{[\ell]})$, one concludes from Lemma~\ref{Lemm-PSC} that $\rho\left(A_{i_k} \dotsm A_{i_1}\right)^{1/k} = \rhod(\mathcal A)$, as required.

Let $\mathcal I_i$ be defined for $i \in \llbracket 1, R\rrbracket$ as in Lemma~\ref{LemmDecomposeCycle} and let $j \in \llbracket 1, R\rrbracket$ be such that $\alpha_j > 0$. Thanks to Lemma~\ref{LemmDecomposeCycle}, there exists a $(\nu, P)$-cycle $(i_1, \dotsc, i_k)$ with $i_1, \dotsc, i_k$ in $\mathcal I_j$. Then, by \eqref{RhoDGeqRhoWord}, \eqref{InegLambdaPNuJ}, and \ref{MainTheoFixedP-Spr}, we have
\[
\rhod(\mathcal A^{[j]}) \leq \rhod(\mathcal A) = \rho(A_{i_k} \dotsm A_{i_1})^{1/k} \leq \rhod(\mathcal A^{[j]}).
\]
In particular, $\rhod(\mathcal A) = \rhod(\mathcal A^{[j]})$ and $\rho(A_{i_k} \dotsm A_{i_1})^{1/k} = \rhod(\mathcal A^{[j]})$. Lemma~\ref{Lemm-PSC} applied to $P_j$ and $\mathcal A^{[j]}$ yields that $\rhop(\nu^{[j]}, P_j, \mathcal A^{[j]}) = \rhod(\mathcal A^{[j]})$. Hence $\rhop(\nu^{[j]}, P_j,\allowbreak \mathcal A^{[j]}) = \rhod(\mathcal A)$, and, since this holds for every $j \in \llbracket 1, R\rrbracket$ such that $\alpha_j > 0$, it follows from \eqref{LambdaPConvexInNu} that $\rhop(\nu, P, \mathcal A) = \rhod(\mathcal A)$, as required.
\end{proof}

\begin{remk}
Theorem~\ref{MainTheoFixedP} and Lemmas~\ref{Lemm-PSC-AIrreducible} and \ref{Lemm-PSC} characterize equality between deterministic and probabilistic joint spectral radii in terms of $P$-cycles and $(\nu, P)$-cycles only, and hence only on $\ceil*{P}$ and $\ceil{\nu}$ (see Remark~\ref{RemkStronglyConnected}). In other words, equality in Theorem~\ref{MainTheoFixedP}\ref{MainTheoFixedP-Equality} depends only on the graph associated with the Markov chain and the possible choices of initial states, but not on the precise values of the non-zero initial and transition probabilities.
\end{remk}

\subsection{Geometric characterization of equality between \texorpdfstring{$\rhod(\mathcal A)$}{rho d(A)} and \texorpdfstring{$\rhop(P,\allowbreak \mathcal A)$}{rho p(P, A)}}
\label{SecGeometric}

It is clear from Theorem \ref{MainTheoFixedP} that equality between $\rhod(\mathcal A)$ and $\rhop(P, \mathcal A)$ is possible only for restricted choices of $\mathcal A$. The goal of this section is to provide a more precise description of such choices of $\mathcal A$ using results from \cite{Protasov2017Matrix}, where the authors classify matrix semigroups of constant spectral radius. We start with the following proposition.

\begin{prop}
\label{PropOrthogonal1}
Let $P \in \mathcal M_N(\mathbb R)$ be a stochastic strongly connected matrix and $\mathcal A = (A_1, \dotsc,\allowbreak A_N) \allowbreak \in \mathcal M_d(\mathbb R)^N$ be such that $\rhod(\mathcal A) = \rhop(P, \mathcal A)$. Assume that there exists $s \in \llbracket 1, N\rrbracket$ such that $C(P, s)$ is irreducible. Then there exists an invertible matrix $G \in \mathcal M_d(\mathbb R)$ such that, for every $P$-cycle $i$ starting at $s$, either $A(i)$ is singular or $\rhod(\mathcal A)^{-k} G A(i) G^{-1}$ is orthogonal, where $k$ is the length of $i$.
\end{prop}

\begin{proof}
We only have to provide an argument if there exists a $P$-cycle $i_\ast$ starting at $s$ such that $A(i_\ast)$ is invertible. In that case, from \eqref{RhoDGeqRhoWord}, $\rhod(\mathcal A) \geq \rho(A(i_\ast))^{1/k_\ast} > 0$, where $k_\ast$ is the length of $i_\ast$. From Lemma~\ref{Lemm-PSC}, the set
\[\{\rhod(\mathcal A)^{-k} A(i) \mid k \in \mathbb N,\; i \text{ is a $P$-cycle starting at $s$ of }\allowbreak\text{length $k$}\}\]
is a matrix semigroup with constant spectral radius. Since, moreover, this semigroup is also irreducible, the conclusion follows from \cite[Theorem 2]{Protasov2017Matrix}.
\end{proof}

\begin{remk}
As remarked in \cite{Protasov2017Matrix}, the problem of classifying matrix semigroups with constant spectral radius is highly nontrivial when the semigroup contains singular matrices. By using additional results from \cite{Protasov2017Matrix}, we may obtain, under the assumptions of Proposition \ref{PropOrthogonal1}, properties on $\rhod(\mathcal A)^{-k} G A(i) G^{-1}$ that are weaker than orthogonality but apply to all matrices $A(i) \in C(P, s)$, and not only nonsingular ones. We refer the interested reader to \cite[Theorem 3 and Corollary 6]{Protasov2017Matrix}.
\end{remk}

%\begin{remk}
%In general, the matrix $G$ in the statement of Proposition \ref{PropOrthogonal1} cannot be chosen independently of $s$. Indeed, take $N = 2$, $P = \begin{pmatrix}0 & 1 \\ 1 & 0 \\ \end{pmatrix}$, and choose $A_1, A_2 \in \mathcal M_2(\mathbb R)$ such that $A_2 A_1$ is not orthogonal and $A_1 A_2 = R_\theta$, the rotation of angle $\theta$, with $\frac{\theta}{\pi}$ irrational. In that case, $C(P, 1)$ and $C(P, 2)$ are the semigroups generated by $A_2 A_1$ and $A_1 A_2$, respectively. Then there does not exist a common basis where $A_1 A_2$ and $A_2 A_1$ are orthogonal. To see that, assume by contradiction that there exists an invertible matrix $G$ such that both $G^{-1} A_2 A_1 G$ and $G^{-1} R_{\theta} G$ are orthogonal. The latter condition implies that $G^{-1} U G$ is orthogonal for every orthogonal matrix $U$, implying that $G = \rho \id_2$ for some $\rho \neq 0$. Hence $A_2 A_1$ is orthogonal, which gives the desired contradiction.
%\end{remk}

A limitation of Proposition \ref{PropOrthogonal1} lies in the fact that, in general, given a stochastic and strongly connected matrix $P$, it is a nontrivial task to verify the existence of an index $s$ such that $C(P, s)$ is irreducible, even if $\mathcal A$ is itself irreducible. However, this is true if one assumes in addition that $\mathcal A$ contains only invertible matrices and that all diagonal elements of $P$ are positive, in which case we have the following proposition.

\begin{prop}
\label{PropPPositiveDiagonal}
Let $P \in \mathcal M_d(\mathbb R)$ be a stochastic strongly connected matrix with positive diagonal entries and $\mathcal A = (A_1, \dotsc,\allowbreak A_N) \allowbreak \in \mathcal M_d(\mathbb R)^N$ be irreducible with $A_1, \dotsc, A_N$ invertible. Then, for every $s \in \llbracket 1, N\rrbracket$, $C(P, s)$ is irreducible. Moreover, $\rhod(\mathcal A) = \rhop(P, \mathcal A)$ if and only if there exists an invertible matrix $G \in \mathcal M_d(\mathbb R)$ such that, for every $i \in \llbracket 1, N\rrbracket$, $\rhod(\mathcal A)^{-1} G A_i G^{-1}$ is orthogonal.
\end{prop}

\begin{proof}
Let $s \in \llbracket 1, N\rrbracket$ and consider the group $\widetilde C(P, s)$ generated by $C(P, s)$. We claim that $A_1, \dotsc, A_N \in \widetilde C(P, s)$. Indeed, since $P$ is strongly connected, there exists a $P$-cycle $i = (i_1, \dotsc, i_k)$ starting at $s$ such that $\{i_1, \dotsc, i_k\} = \llbracket 1, N\rrbracket$. Since $p_{i_k i_k} > 0$, then $A_{i_k}^2 A_{i_{k-1}} \dotsm A_{i_1} \in C(P, s)$ and
\[
A_{i_k} = \left(A_{i_k}^2 A_{i_{k-1}} \dotsm A_{i_1}\right) \left(A_{i_k} \dotsm A_{i_1}\right)^{-1} \in \widetilde C(P, s).
\]
Similarly, since $p_{i_{k-1} i_{k-1}} > 0$, then $A_{i_k} A_{i_{k-1}}^2 A_{i_{k-2}} \dotsm A_{i_1} \in C(P, s)$ and
\[
A_{i_{k-1}} = A_{i_k}^{-1} \left(A_{i_k} A_{i_{k-1}}^2 A_{i_{k-2}} \dotsm A_{i_1}\right) \left(A_{i_k} A_{i_{k-1}} \dotsm A_{i_1}\right)^{-1} A_{i_k} \in \widetilde C(P, s).
\]
An inductive reasoning based on the identity
\begin{equation}
\label{InductionPositiveDiagonal}
A_{i_j} = \left(A_{i_k} \dotsm A_{i_{j+1}}\right)^{-1} \left(A_{i_k} \dotsm A_{i_{j+1}} A_{i_j}^2 A_{i_{j-1}} \dotsm A_{i_1}\right) \left(A_{i_k} \dotsm A_{i_1}\right)^{-1} \left(A_{i_k} \dotsm A_{i_{j+1}}\right)
\end{equation}
yields that $A_{i_j} \in \widetilde C(P, s)$ for $j \in \llbracket 1, k\rrbracket$, as required.

To prove that $C(P, s)$ is irreducible for every $s$, assume by contradiction that there exists $s \in \llbracket 1, N\rrbracket$ such that $C(P, s)$ is reducible. Then the group $\widetilde C(P, s)$ is also reducible, however, since it contains $A_1, \dotsc, A_N$, this contradicts the irreducibility of $\mathcal A$.

Since 
$A_1, \dotsc, A_N$ are 
%$\mathcal A$ is made of 
invertible matrices, $\rhod(\mathcal A)$ is positive and, with no loss of generality, we can assume that $\rhod(\mathcal A) = 1$. If $\rhod(\mathcal A) = \rhop(P, \mathcal A)$, then, applying Proposition \ref{PropOrthogonal1} to $C(P, 1)$, there exists a basis in which every $M \in C(P, 1)$ is orthogonal. Hence, in this same basis, $\widetilde C(P, 1)$ is also made of orthogonal matrices, yielding the conclusion. On the other hand, if there exists a basis in which $A_1, \dotsc, A_N$ are orthogonal, then $\rho(A(i)) = 1$ for every $P$-word $i$, and the conclusion follows by Lemma~\ref{Lemm-PSC}.
\end{proof}

\begin{remk}
\label{RemkPositiveDiagonal}
Notice that, to obtain the second part of the conclusion of Proposition \ref{PropPPositiveDiagonal}, it is enough that there exists $s \in \llbracket 1, N\rrbracket$ such that $C(P, s)$ is irreducible and the generated group $\widetilde C(P, s)$ contains all matrices $A_1, \dotsc, A_N$. The assumption that $P$ has positive diagonal entries is used to guarantee the latter, and therefore it can be replaced by any other condition ensuring that $A_1, \dotsc, A_N$ belong to $\widetilde C(P, s)$ for some $s \in \llbracket 1, N\rrbracket$. For instance, assume that $p_{11} = 0$ and $p_{jj} > 0$ for $j \in \llbracket 2, N\rrbracket$. For every $P$-cycle $(i_1, \dotsc, i_k)$ with $i_1 = 1$ and $i_j \neq 1$ for every $j \in \llbracket 2, k\rrbracket$, we can proceed as in the proof of the proposition to obtain that $A_{i_j} \in \widetilde C(P, 1)$ for every $j \in \llbracket 2, k\rrbracket$ and use the identity
\[
A_{i_1} = \left(A_{i_k} \dotsm A_{i_2}\right)^{-1} \left(A_{i_k} \dotsm A_{i_1}\right)
\]
to obtain that $A_{i_1} \in \widetilde C(P, 1)$. Since $P$ is strongly connected, every matrix $A_i$, $i \in \llbracket 1, N\rrbracket$, belongs to such a $P$-cycle, hence the conclusion.
\end{remk}

% \color{blue}
% MS: I find that it would be more coherent to include the counterexample of the remark below in the discussion of the remark above. It is true that it would become a too long remark: is it possible to present it in a more organized way? Two related claims: (i) some diagonal elements may be zero, but (ii) the fact that one only is nonzero is not enough in general. Maybe is not necessary to pass through primitivity, at least in the statement

%\begin{remk}
At the light of Remark~\ref{RemkPositiveDiagonal}, we may wonder whether the second part of the conclusion of Proposition~\ref{PropPPositiveDiagonal} can be obtained under even weaker assumptions on the matrix $P$, allowing for instance the presence of more than one diagonal element equal to zero, but requiring at least one non-zero element in the diagonal. The example given below shows that this is not possible.
%such as the assumption that $P$ is primitive (i.e., that there exists $k \in \mathbb N$ such that all entries of $P^k$ are positive) or that $P$ has at least one diagonal entry which is non-zero. These assumptions, however, are not sufficient to obtain the second part of the conclusion of Proposition~\ref{PropPPositiveDiagonal}. To see that, 

\begin{expl}
Consider the case $d = 2$, $N = 3$,
\[
A_1 = \id, \quad A_2 = \begin{pmatrix}
0 & 1 \\
-1 & 0 \\
\end{pmatrix}, \quad A_3 = \begin{pmatrix}
0 & -\frac{1}{2} \\
1 & 0 \\
\end{pmatrix}, \quad P = \begin{pmatrix}
\frac{1}{2} & \frac{1}{2} & 0 \\
0 & 0 & 1 \\
1 & 0 & 0 \\
\end{pmatrix}.
\]
Note that, in this case,
\[
A_3 A_2 = \begin{pmatrix}
\frac{1}{2} & 0 \\
0 & 1 \\
\end{pmatrix}, \quad A_2 A_3 = \begin{pmatrix}
1 & 0 \\
0 & \frac{1}{2} \\
\end{pmatrix}.
\]
The matrix $P$ is stochastic, strongly connected, and its unique invariant probability is $\nu = (\frac{1}{2}, \frac{1}{4}, \frac{1}{4})$. Moreover, $\mathcal A = (A_1, A_2, A_3)$ is irreducible and $A_1, A_2, A_3$ are invertible. Denoting by $\norm{\cdot}$ the Euclidean norm in $\mathbb R^2$, we have $\norm{A_1} = \norm{A_2} = \norm{A_3} = 1$, yielding that $\rhod(\mathcal A) \leq 1$, and we easily check that $\rhod(\mathcal A) = 1$ by considering $\sigma \in \mathfrak S$ given by $\sigma(i) = 1$ for every $i \in \mathbb N$. Moreover, for any $(\nu, P)$-word $(i_1, \dotsc, i_k)$, there exist an integer $\ell \geq 0$ and $a, b \in \{0, 1\}$ such that $\norm{A_{i_k} \dotsm A_{i_1}} = \norm{A_2^b (A_3 A_2)^\ell A_3^a}$. Setting $x = \begin{pmatrix}1\\0\end{pmatrix}$ in the case $a = 1$ and $x = \begin{pmatrix}0\\1\end{pmatrix}$ in the case $a = 0$, it is immediate to verify that $\norm{A_2^b (A_3 A_2)^\ell A_3^a x} = 1$, yielding that $\norm{A_{i_k} \dotsm A_{i_1}} = 1$ for every $(\nu, P)$-word $(i_1, \dotsc, i_k)$, and thus $\rhop(\nu, P, \mathcal A) = \rhop(P, A) = 1$. However, $A_3$ is not similar to an orthogonal matrix, and hence the second conclusion of Proposition~\ref{PropPPositiveDiagonal} does not hold. Notice moreover that, in this case, $C(P, 1) = \{(A_3 A_2)^n \mid n \in \mathbb N \cup \{0\}\}$, $C(P, 2) = \{(A_3 A_2)^n \mid n \in \mathbb N\}$, and $C(P, 3) = \{(A_2 A_3)^n \mid n \in \mathbb N\}$, and thus $C(P, s)$ is reducible for every $s \in \{1, 2, 3\}$.
\end{expl}

\begin{remk}
We now provide a description of all cases where equality holds between $\rhod(\mathcal A)$ and $\rhop(\nu, P, \mathcal A)$ under the assumption that $\mathcal A$ is irreducible and made of two invertible matrices.
\begin{enumerate}
\item If $P = \begin{pmatrix}p & 1 - p \\ 1 - q & q\end{pmatrix}$ for $p, q \in [0, 1)$ with $p + q > 0$, by Remark \ref{RemkPositiveDiagonal}, equality occurs if and only if there exists an invertible matrix $G \in \mathcal M_d(\mathbb R)$ such that $\rhod(\mathcal A)^{-1} G A_1 G^{-1}$ and $\rhod(\mathcal A)^{-1} G A_2 G^{-1}$ are orthogonal.
\item If $P = \begin{pmatrix}0 & 1 \\ 1 & 0\end{pmatrix}$, equality occurs if and only if $\rho(A_1 A_2) = \rho(A_2 A_1) = \rhod(\mathcal A)^2$.
\item If $P = \id_2$, equality occurs if and only if $\rho(A_i) = \rhod(\mathcal A)$ whenever $\nu_i > 0$, $i \in \{1, 2\}$.
\item If $P = \begin{pmatrix}1 & 0 \\ 1 - p & p\end{pmatrix}$ for some $p \in [0, 1)$, then equality is equivalent to $\rho(A_1) = \rhod(\mathcal A)$.
\item If $P = \begin{pmatrix}p & 1 - p\\ 0 & 1\end{pmatrix}$ for some $p \in [0, 1)$, then equality is equivalent to $\rho(A_2) = \rhod(\mathcal A)$.
\end{enumerate}
\end{remk}

%\begin{lemm}
%Let $\mathcal A = (A_1, \dotsc, A_N) \in \mathcal M_d(\mathbb R)^N$ be irreducible and $c > 0$. The following assertions are equivalent.
%\begin{enumerate}
%\item\label{CaseA} For every sequence $(i_j)_{j \in \mathbb N}$ in $\llbracket 1, N\rrbracket$, one has
%\[
%\norm*{A_{i_k} \dotsm A_{i_1}}_{\mathrm B}^{1/k} = c.
%\]
%\item\label{CaseB} For every sequence $(i_j)_{j \in \mathbb N}$ in $\llbracket 1, N\rrbracket$, one has
%\[
%\rho(A_{i_k} \dotsm A_{i_1})^{1/k} = c.
%\]
%\end{enumerate}
%\end{lemm}
%
%\begin{proof}[Sketch of the proof]
%\begin{itemize}
%\item $\Rightarrow$: immediate by Gelfand's formula.
%
%\item $\Leftarrow$: with no loss of generality, we assume $c = 1$ and we argue by contradiction. By \ref{CaseB}, taking the $\sup$ over all words of length $k$ and letting $k \to \infty$, one obtains by the Joint Spectral Radius Theorem that the joint spectral radius of $\mathcal A$ is $1$. If \ref{CaseA} does not hold, then there exist $i_1, \dotsc, i_k$ such that $\norm*{A_{i_k} \dotsm A_{i_1}}_{\mathrm B}^{1/k} < 1$ for some $(i_j)_{j \in \mathbb N}$. Then $\rho(A_{i_k} \dotsm A_{i_1}) \leq \norm*{A_{i_k} \dotsm A_{i_1}}_{\mathrm B}^{1/k} < 1$, contradicting \ref{CaseB}.
%\end{itemize}
%\end{proof}

\subsection{Equality between \texorpdfstring{$\rhod(\mathcal A)$}{rho d(A)} and \texorpdfstring{$\rhop(\mathcal A)$}{rho p(A)}}
\label{SecEquality}

Based on the results obtained previously, we can now address the issue of characterizing the equality between $\rhod(\mathcal A)$ and $\rhop(\mathcal A)$. Recall that the latter is defined as the maximum of $\rhop(\nu, P, \mathcal A)$ over all pairs $(\nu, P)$.

\begin{theo}
\label{MainTheo}
Let $\mathcal A = (A_1, \dotsc, A_N) \in \mathcal M_d(\mathbb R)^N$. Then the following statements are equivalent:
\begin{enumerate}
\item\label{MainTheo-Equality} $\rhod(\mathcal A) = \rhop(\mathcal A)$.
\item\label{MainTheo-Spr} There exist $i_1, \dotsc, i_k \in \llbracket 1, N\rrbracket$ pairwise distinct such that
\begin{equation}
\label{NSCForEquality}
\rhod(\mathcal A) = \rho(A_{i_k} \dotsm A_{i_1})^{1/k}.
\end{equation}
\end{enumerate}
\end{theo}

\begin{proof}
We start by proving that \ref{MainTheo-Equality} implies \ref{MainTheo-Spr}. Recall that, by Remark \ref{RemkSupAttained}, there exist a stochastic matrix $P$ and an invariant probability $\nu$ for $P$ such that $\rhop(\nu, P, \mathcal A) = \rhop(\mathcal A)$. Using \ref{MainTheo-Equality}, we deduce that $\rhop(\nu, P, \mathcal A) = \rhod(\mathcal A)$. It is clear that there exists a $(\nu, P)$-cycle $(i_1, \dotsc, i_k)$ such that $i_1, \dotsc, i_k$ are pairwise distinct, and the conclusion follows from Theorem \ref{MainTheoFixedP}.

To prove that \ref{MainTheo-Spr} implies \ref{MainTheo-Equality}, let $P = (p_{ij})$ be a stochastic matrix with $p_{i_{j-1} i_j} = 1$ for $j \in \llbracket 2, k \rrbracket$ and $p_{i_k i_1} = 1$. Set $\nu \in \mathbb R^N$ as the probability vector such that $\nu_{i_j} = \frac{1}{k}$ for $j \in \llbracket 1, k\rrbracket$. Then $\nu$ is invariant under $P$ and the set of $(\nu, P)$-cycles is made of the shifts of $(i_1, \dotsc, i_k)$ and their powers. Moreover, for every such $(\nu, P)$-cycle $(j_1, \dotsc, j_s)$, we have
\[
\rho(A_{j_s} \dotsm A_{j_1})^{1/s} = \rho(A_{i_k} \dotsm A_{i_1})^{1/k} = \rhod(\mathcal A).
\]
Indeed, this follows from the fact that $\rho(M_1 M_2) = \rho(M_2 M_1)$ for every $M_1, M_2 \in \mathcal M_d(\mathbb R)$. Then Theorem \ref{MainTheoFixedP}\ref{MainTheoFixedP-Spr} holds, hence $\rhop(\nu, P, \mathcal A) = \rhod(\mathcal A)$, and the conclusion follows from \eqref{Inequality}.
\end{proof}

\begin{remk}
It follows from \eqref{Inequality} that, if $\rhod(\mathcal A) > 0$, the ratio $\frac{\rhop(\mathcal A)}{\rhod(\mathcal A)}$ belongs to $[0, 1]$ and Theorem \ref{MainTheo} addresses the case where it is equal to $1$. We provide next an example where it is equal to $0$, proving that it is not possible to find a uniform positive lower bound for this ratio. Indeed, considering
\[
A_1 = 
\begin{pmatrix}
0 & 1 & 0 \\
0 & 0 & 1 \\
0 & 0 & 0 \\
\end{pmatrix}, \qquad A_2 = \begin{pmatrix}
0 & 0 & 0 \\
0 & 0 & 0 \\
1 & 0 & 0 \\
\end{pmatrix},
\]
an immediate computation yields
\[
A_1^2 A_2 = \begin{pmatrix}1 & 0 & 0 \\ 0 & 0 & 0 \\ 0 & 0 & 0 \\\end{pmatrix},\;
A_1 A_2 A_1 = \begin{pmatrix}0 & 0 & 0 \\ 0 & 1 & 0 \\ 0 & 0 & 0 \\\end{pmatrix},\;
A_2 A_1^2 = \begin{pmatrix}0 & 0 & 0 \\ 0 & 0 & 0 \\ 0 & 0 & 1 \\\end{pmatrix},
\]
and $A_1^3 = A_1 A_2^2 = A_2 A_1 A_2 = A_2^2 A_1 = A_2^3 = 0$. Let $\norm{\cdot}_1$ denote the matrix norm induced by the $\ell^1$ norm in $\mathbb R^3$. Define
\[
\mathcal E = \{(2, 1, 1, 2, 1, 1, \dotsc), (1, 2, 1, 1, 2, 1, \dotsc), (1, 1, 2, 1, 1, 2, \dotsc)\}
\]
and, for $k \in \mathbb N$, let $\mathcal E_k$ be the set made of the three words of length $k$ obtained by taking the first $k$ entries of each element of $\mathcal E$. By an easy computation, we get that, for every $k \in \mathbb N$ and $(i_1, \dotsc, i_k) \in \llbracket 1, N\rrbracket^k$,
\[
\norm{A_{i_k} \dotsm A_{i_1}}_{1} = 
\begin{dcases*}
1, & if $(i_1, \dotsc, i_k)\in \mathcal E_k$, \\
0, & otherwise.
\end{dcases*}
\]
We then obtain that $\rhod(\mathcal A) = 1$. On the other hand, for every stochastic matrix $P \in \mathcal M_2(\mathbb R)$ and every invariant probability vector $\nu$ for $P$, we have $\mathbb P_{(\nu, P)}(\mathcal E) = 0$. Hence
\[
\lim_{n \to \infty} \norm{A_{i_{n}} \dotsm A_{i_1}}_{1}^{1/n} = 0 \quad \mathbb P_{(\nu, P)}\text{-a.s.},
\]
proving that $\rhop(\nu, P, \mathcal A) = 0$. Then $\rhop(\mathcal A) = 0$.
\end{remk}

\section{Markov chains of higher order}

In this section, we extend the previous results to probability measures on $\mathfrak S$ obtained from discrete-time shift-invariant Markov chains of order $m \geq 1$. Any such probability measure $\mu$ can be described by a pair $(\nu, P)$ 
%, where $\nu$ and $P$ are 
of
tensors of orders $m$ and $m+1$, respectively, where the non-negative scalar
%number
$P_{i_1 \dotso i_m i_{m+1}}$ represents the probability to switch from the state $i_m$ to the state $i_{m+1}$ when the previous $m$ states of the chain are $(i_1, \dotsc, i_m)$, and $\nu_{i_1 \dotso i_m}$ represents the probability of the first $m$ states being $(i_1, \dotsc, i_m)$. In particular, for every $(i_1, \dotsc, i_m) \in \llbracket 1, N\rrbracket^m$, we have that
\[
\sum_{i_{m+1} = 1}^N P_{i_1 \dotso i_m i_{m+1}} = 1
\]
and $\nu$ satisfies
\[
\sum_{(i_1, \dotsc, i_m) \in \llbracket 1, N\rrbracket^m} \nu_{i_1 \dotso i_m} = 1.
\]
We refer to such $\nu$ and $P$ as a \emph{probability tensor of order $m$} and a \emph{stochastic tensor of order $m+1$}, respectively. The shift-invariance property now reads
\[
\sum_{i_1=1}^N \nu_{i_1 \dotso i_m} P_{i_1 \dotso i_{m+1}} = \nu_{i_2 \dotso i_{m+1}}, \qquad \text{ for every } (i_2, \dotsc, i_{m+1}) \in \llbracket 1, N\rrbracket^m,
\]
and any probability tensor $\nu$ satisfying the above shift-invariant property is said to be \emph{invariant} under $P$. The probabilistic joint spectral radius $\rhop(\nu, P, \mathcal A)$ associated with $(\nu, P)$ is still defined by \eqref{EqDefiLambdaP}, where the expectation $\mathbb E_{(\nu, P)}$ corresponds to the probability measure on $\mathfrak S$ defined above.

Markov chains of order $m \geq 1$ can be canonically transformed into Markov chains of order $1$ by considering as state space the set $\llbracket 1, N\rrbracket^m$ and defining a pair $(\widehat \nu, \widehat P)$ from $(\nu, P)$ by $\widehat\nu_{(i_1, \dotsc, i_m)} = \nu_{i_1 \dotso i_m}$ and
\[\widehat P_{(i_1, \dotsc, i_m), (j_1, \dotsc, j_{m})} = 
\begin{dcases*}
P_{i_1 \dotso i_{m} j_m} & if $(i_2, \dotsc, i_m) = (j_1, \dotsc, j_{m-1})$, \\
0 & otherwise,
\end{dcases*}
\]
for every $(i_1, \dotsc, i_m)$ and $(j_1, \dotsc, j_m)$ in $\llbracket 1, N\rrbracket^{m}$. It is immediate from the definitions and the shift-invariance property that
\[\rhop(\nu, P, \mathcal A) = \rhop(\widehat\nu, \widehat P, \widehat{\mathcal A}),\]
where $\widehat{\mathcal A} = (\widehat A_{i_1 \dotso i_m})_{(i_1, \dotsc, i_m) \in \llbracket 1, N\rrbracket^m}$ and $\widehat A_{i_1 \dotso i_m} = A_{i_m}$ for every $(i_1, \dotsc, i_m) \in \llbracket 1, N\rrbracket^m$.

For every positive integer $k$, we say that $(i_1, \dotsc, i_k)$ is a \emph{$(\nu, P)$-cycle} if
\[
\bigl((i_{-m+2}, \dotsc, i_0, i_1), \dotsc, (i_{k-m+1}, \dotsc, i_k)\bigr)
\]
is a $(\widehat\nu, \widehat P)$-cycle, where $z \mapsto i_z$ is extended to $\mathbb Z$ by $k$-periodicity.

Applying Theorem~\ref{MainTheoFixedP} to $(\widehat\nu, \widehat P)$ and $\widehat{\mathcal A}$, we deduce at once the following.
\begin{theo}
\label{TheoEquivalenceOrderM}
Let $m$ be a positive integer, $P$ be a stochastic tensor of order $m+1$, $\nu$ be an invariant probability tensor for $P$, and $\mathcal A = (A_1,\allowbreak \dotsc,\allowbreak A_N) \in \mathcal M_d(\mathbb R)^N$. Then the following statements are equivalent:
\begin{enumerate}
\item $\rhod(\mathcal A) = \rhop(\nu, P, \mathcal A)$.
\item\label{TheoEquivalenceOrderM-B} $\rho(A_{i_k} \dotsm A_{i_1})^{1/k} = \rhod (\mathcal A)$ for every $(\nu, P)$-cycle $(i_1, \dotsc, i_k)$.
\end{enumerate}
\end{theo}

Recall that \eqref{SwitchedSystem} is said to be \emph{periodically stable} if $\rho(\sigma) < 1$ for all periodic signals $\sigma \in \mathfrak S$. It has been shown in \cite{Dai2011Periodically} that this property implies $\rhop(\nu, P, \mathcal A) < 1$ for every strongly connected stochastic matrix $P \in \mathcal M_N(\mathbb R)$, where $\nu \in \mathbb R^N$ is the unique invariant probability vector for $P$. A slightly improved version of this result can be obtained as a consequence of Theorem~\ref{TheoEquivalenceOrderM} as stated in the following corollary.

\begin{coro}
Assume that \eqref{SwitchedSystem} is periodically stable. Then, for every $m \in \mathbb N$, every stochastic tensor $P$ of order $m+1$, and every invariant probability tensor $\nu$ for $P$,
we have $\rhop(\nu, P, \mathcal A) < 1$.
%, for every $(m, \nu, P)$, 
%where $m \in \mathbb N$, $P$ is a stochastic tensor of order $m+1$, and $\nu$ is an invariant probability tensor for $P$
\end{coro}

\begin{proof}
By the Joint Spectral Radius Theorem (see, e.g., \cite[Theorem 2.3]{Jungers2009Joint}), periodic stability implies that $\rhod(\mathcal A) \leq 1$. In the case $\rhod(\mathcal A) < 1$, the conclusion follows immediately. Otherwise, when $\rhod(\mathcal A) = 1$, the periodic stability assumption implies that assertion \ref{TheoEquivalenceOrderM-B} from Theorem~\ref{TheoEquivalenceOrderM} does not hold, which proves that $\rhop(\nu, P, \mathcal A) < \rhod(\mathcal A) = 1$, yielding the conclusion.
\end{proof}

Similarly %to
as for Theorem~\ref{TheoEquivalenceOrderM}, we deduce by applying Theorem~\ref{MainTheo} to $(\widehat\nu, \widehat P)$ and $\widehat{\mathcal A}$ the following.
\begin{theo}
\label{MainTheoOrderM}
Let $m$ be a positive integer and $\mathcal A = (A_1, \dotsc, A_N) \in \mathcal M_d(\mathbb R)^N$. Then the following statements are equivalent:
\begin{enumerate}
\item\label{MainTheoOrderM-Equality} $\rhod(\mathcal A) = \rhop(m, \mathcal A)$, where $\rhop(m, \mathcal A)$ is the supremum of $\rhop(\nu, P, \mathcal A)$ over all pairs $(\nu, P)$ with $P$ a stochastic tensor of order $m+1$ and $\nu$ an invariant probability tensor for $P$.
\item\label{MainTheoOrderM-Spr} There exist $i_1, \dotsc, i_k \in \llbracket 1, N\rrbracket$ such that
\begin{equation*}
\rhod(\mathcal A) = \rho(A_{i_k} \dotsm A_{i_1})^{1/k}
\end{equation*}
and $(i_{j_1}, \dotsc, i_{j_1 + m - 1}) \neq (i_{j_2}, \dotsc, i_{j_2 + m - 1})$ whenever $j_1, j_2 \in \llbracket 1, k\rrbracket$ with $j_1 \neq j_2$, where $z \mapsto i_z$ is extended to $\mathbb Z$ by $k$-periodicity.
\end{enumerate}
\end{theo}

%\begin{remk}
%Let $\mathcal A = (A_1, \dotsc, A_N)$. For every $\ell \in \mathbb N$ and word $w = (i_1, \dotsc, i_\ell) \in \llbracket 1, N\rrbracket^\ell$, set $A(w) = A_{i_\ell} \dotsm A_{i_1}$ and let $\abs{w} = \ell$ be the length of $w$. Notice that, by proceeding similarly to the second part of the proof of Theorem~\ref{MainTheo}, one can construct, for every word $w$ of finite length, a Markov chain of order $\abs{w}$ with tensors $\nu_w$, $P_w$ such that $\rho(A(w))^{1/\abs{w}} = \rhop(\nu_w, P_w, \mathcal A)$. One deduces 
%that
%\[
%\rhod(\mathcal A)=\sup_{w\text{ word of finite length}}\rho(A(w))^{1/\abs{w}} \le \sup_{m \in \mathbb N} \rhop(m, \mathcal A),
%\]
%where the equality is a consequence of the Joint Spectral Radius Theorem (see, e.g., \cite{Jungers2009Joint}). Since, moreover, $\rhop(m, \mathcal A)\le \rhod(\mathcal A)$ for every $m$, it follows that 
%$\rhod(\mathcal A)=\sup_{m \in \mathbb N} \rhop(m, \mathcal A)$.
%%the following equalities:
%%\[
%%\sup_{w\text{ word of finite length}}\rho(A(w))^{1/\abs{w}} = \rhod(\mathcal A) = \sup_{m \in \mathbb N} \rhop(m, \mathcal A).
%%\]
%%Indeed, the left equality is a consequence of the Joint Spectral Radius Theorem (see, e.g., \cite{Jungers2009Joint}), while the second equality follows from the previous arguments.
%\end{remk}

As a consequence of Theorem~\ref{MainTheoOrderM}, we have the following corollary. To state it, recall that $\mathcal A$ is said to have the \emph{finiteness property} if there exist $i_1, \dotsc, i_k \in \llbracket 1, N\rrbracket$ such that $\rhod(\mathcal A) = \rho(A_{i_k} \dotsm A_{i_1})^{1/k}$.

\begin{coro}
\label{CoroFinitenessProp}
Let $\mathcal A = (A_1, \dotsc, A_N)$. Then $\mathcal A$ has the finiteness property if and only if there exists $m \in \mathbb N$ such that $\rhod(\mathcal A) = \rhop(m, \mathcal A)$.
\end{coro}

\begin{proof}
If there exists $m$ such that $\rhod(\mathcal A) = \rhop(m, \mathcal A)$, then the finiteness property of $\mathcal A$ follows immediately from Theorem~\ref{MainTheoOrderM}. Assume now that $\mathcal A$ has the finiteness property and let $i_1, \dotsc, i_k \in \llbracket 1, N\rrbracket$ be such that $\rhod(\mathcal A) = \rho(A_{i_k} \dotsm A_{i_1})^{1/k}$. Extend $z \mapsto i_z$ over $\mathbb Z$ by $k$-periodicity and let $k^\prime$ be the minimal period of $z \mapsto i_z$. Without loss of generality, we can assume that $k = k^\prime$. We claim that property \ref{MainTheoOrderM-Spr} of Theorem~\ref{MainTheoOrderM} holds with $m = k$. Indeed, let $j_1, j_2 \in \llbracket 1, k\rrbracket$ be such that $(i_{j_1}, \dotsc, i_{j_1+k-1}) = (i_{j_2}, \dotsc, i_{j_2 + k - 1})$ and assume, to obtain a contradiction, that $j_1 \neq j_2$. Without loss of generality, $j_1 < j_2$. Set $k^{\prime\prime} = j_2 - j_1$ and notice that $0 < k^{\prime\prime} < k$ and $i_{j_1 + \ell} = i_{j_1 + k^{\prime\prime} + \ell}$ for every $\ell \in \llbracket 0, k-1\rrbracket$. Since $z \mapsto i_z$ is $k$-periodic, the previous equality holds for every $\ell \in \mathbb Z$, proving that $z \mapsto i_z$ is $k^{\prime\prime}$-periodic, contradicting the minimality of $k$ as period of $z \mapsto i_z$. Hence property \ref{MainTheoOrderM-Equality} of Theorem~\ref{MainTheoOrderM} holds, as required.
\end{proof}

\begin{remk}
Given $\mathcal A = (A_1, \dotsc, A_N)$, $\ell \in \mathbb N$, and a word $w = (i_1, \dotsc, i_\ell) \in \llbracket 1, N\rrbracket^\ell$, set $A(w) = A_{i_\ell} \dotsm A_{i_1}$ and let $\abs{w} = \ell$ be the length of $w$. Notice that, by proceeding similarly to the second part of the proof of Theorem~\ref{MainTheo}, we can construct, for every word $w$ of finite length, a Markov chain of order $\abs{w}$ with tensors $\nu_w$, $P_w$ such that $\rho(A(w))^{1/\abs{w}} = \rhop(\nu_w, P_w, \mathcal A)$. We deduce that
\[
\rhod(\mathcal A)=\sup_{w\text{ word of finite length}}\rho(A(w))^{1/\abs{w}} \le \sup_{m \in \mathbb N} \rhop(m, \mathcal A),
\]
where the equality is a consequence of the Joint Spectral Radius Theorem (see, e.g., \cite{Jungers2009Joint}). Since, moreover, $\rhop(m, \mathcal A)\le \rhod(\mathcal A)$ for every $m$, it follows that 
$\rhod(\mathcal A)=\sup_{m \in \mathbb N} \rhop(m, \mathcal A)$.

A further characterization of the equivalence in Corollary~\ref{CoroFinitenessProp} can then be stated as follows: an $N$-tuple of matrices $\mathcal A = (A_1, \dotsc, A_N)$ satisfies the finiteness property if and only if
\[
\sup_{m, \nu, P} \rhop(\nu, P, \mathcal A)
\]
is attained at some $(m, \nu, P)$, where the supremum is taken over all $(m, \nu, P)$ with $m \in \mathbb N$, $P$ a stochastic tensor of order $m+1$, and $\nu$ an invariant probability tensor for $P$.
\end{remk}

\bibliographystyle{abbrv}
\bibliography{Bib}

\end{document}